\documentclass[12pt]{article}
\usepackage{eurosym}
\usepackage{amsfonts}
\usepackage{graphicx}
\usepackage{amsmath}

\numberwithin{equation}{section}

\newtheorem{theorem}{Theorem}[section]

\newtheorem{corollary}[theorem]{Corollary}

\newtheorem{lemma}[theorem]{Lemma}

\newtheorem{proposition}[theorem]{Proposition}

\newenvironment{proof}[1][Proof]{\textbf{#1.} }{\ \rule{0.5em}{0.5em}}
\input{tcilatex}
\begin{document}

\title{Quasi-sure convergence theorem in $p$-variation distance for
stochastic differential equations}
\author{H. Boedihardjo\thanks{%
Mathematical Institute, University of Oxford, Oxford OX1 3LB, England.%
\newline
Email: Horatio.Boedihardjo@maths.ox.ac.uk } \ and Z. Qian\thanks{%
Mathematical Institute, University of Oxford, Oxford OX1 3LB, England. Email: qianz@maths.ox.ac.uk.}}
\maketitle

\begin{abstract}
In this paper by calculating carefully the capacities (defined by high order
Sobolev norms on the Wiener space) for some functions of Brownian motion, we
show that the dyadic approximations of the sample paths of the Brownian
motion converge in the $p$-variation distance to the Brownian motion except
for a slim set (i.e. except for a zero subset with respect to the capacity
on the Wiener space of any order). This presents a way for studying
quasi-sure properties of Wiener functionals by means of the rough path
analysis.
\end{abstract}

\section{Introduction}

It has been suggested that the recent theory of rough paths, put forward in
T. Lyons \cite{MR1654527}, and developed further over the past years by T.
Lyons with his coauthors, and other authors, see \cite{MR2036784}, should
simplify and strengthen the results in quasi-sure analysis over the Wiener
space, which was initiated by P. Malliavin (see \cite{MR1450093} for
example). In fact, the rough path analysis has direct applications in
solving stochastic differential equations quasi-surely (see below for a
precise meaning). What however is missing in literature is an approximation
theorem towards Brownian motion sample paths by simple random curves in $p$%
-variation distance and in quasi-sure sense, i.e. except for a slim subset,
instead of a probability zero set, see below for a definition of slim
subsets in Wiener space. The main goal of the paper is to establish such a
theorem (see Theorem \ref{th5} and Theorem \ref{th6} below). The results
obtained in the paper allow to study quasi-sure properties for important
Wiener functionals -- solutions of It\^{o}'s stochastic differential
equations. The main step in our proof is the construction of quasi-surely
defined \textit{geometric} rough paths associated with Brownian motion,
which we believe has other applications, though not discussed in the present
paper. It is known that there is an equivalence of capacity zero sets
(defined in terms of Dirichlet norms of the Ornstein-Uhlenbeck process on
the Wiener space) and the polar sets defined by the Brownian sheets.
Therefore, with the quasi-surely defined geometric Brownian motion paths we
constructed, it is possible to define the rough path analysis for the
Brownian sheet, thus provides another possible route to study stochastic
partial differential equations via the rough path analysis, this potential
application however is not pursed further here. There is increasing interest
in applying rough path analysis to the study of stochastic partial
differential equations, such as the recent papers \cite{MR2165581}, \cite%
{MR2510132}, \cite{MR2765508}, \cite{MR2566910}, \cite{MR2255351}, \cite%
{hairer1} and the references therein.

In order to address the question we investigate in this perspective, we
begin with some elements in the analysis of rough paths, and establish the
notions and notations which will be used throughout the paper.

\subsection{Concept of rough paths}

Let $(\mathbb{R}^{d})^{\otimes k}=\mathbb{R}^{d}\otimes \cdots \otimes 
\mathbb{R}^{d}$ be the tensor product of $k$-folds of the Euclidean space $%
\mathbb{R}^{d}$; which may be identified with $\mathbb{R}^{kd}$; equipped
with the corresponding Euclidean norm. For a positive integer $n$, $T^{(n)}(%
\mathbb{R}^{d})$ denotes the truncated tensor algebra 
\begin{equation*}
T^{(n)}(\mathbb{R}^{d})=\sum_{k=0}^{n}\oplus (\mathbb{R}^{d})^{\otimes k}%
\text{ ,}
\end{equation*}%
where $(\mathbb{R}^{d})^{\otimes 0}=\mathbb{R}^{1}$.

Given $T>0$, $\Delta =\left\{ (s,t):0\leq s\leq t\leq T\right\} $. $T>0$
will be fixed but arbitrary, so it will be assumed to be $1$ without lose of
generality.

A \emph{continuous} path $w:[0,T]\rightarrow \mathbb{R}^{d}$ is said to have
finite total variation on $[0,T]$, if 
\begin{equation*}
\sup_{D}\sum_{l}\left\vert w_{t_{l}}-w_{t_{l-1}}\right\vert <\infty
\end{equation*}%
where $\sup_{D}$ takes over all finite partitions of $[0,T]$: 
\begin{equation*}
D=\{0=t_{0}<\cdots <t_{m}=T\}\text{.}
\end{equation*}%
This convention applies to similar situations below without further
qualification.

Let $\Omega ^{\infty }(\mathbb{R}^{d})$ denote the totality of all
continuous paths in $\mathbb{R}^{d}$ with finite total variations on $[0,T]$%
. If $w\in \Omega ^{\infty }(\mathbb{R}^{d})$, the $k$-th iterated path
integral over $[s,t]$ 
\begin{equation}
w_{s,t}^{k}=\dint\limits_{s<t_{1}<\cdots <t_{k}<t}dw_{t_{1}}\otimes \cdots
\otimes dw_{t_{k}}\text{.}  \notag
\end{equation}%
By definition, $w_{s,t}^{1}=w_{t}-w_{s}$ is the increment of the path $w$
over $[s,t]$, and for $k\geq 2$ 
\begin{equation*}
w_{s,t}^{k}=\lim_{m(D)\rightarrow
0}\sum_{l}\sum_{j=1}^{k-1}w_{s,t_{l-1}}^{j}\otimes w_{t_{l-1},t_{l}}^{k-j}
\end{equation*}%
are defined inductively. Collecting all $k$-th iterated integrals (up to
degree $n$) together we define $L_{n}(w):\Delta \rightarrow T^{(n)}(\mathbb{R%
}^{d})$ by 
\begin{equation*}
L_{n}(w)_{s,t}=(1,L_{n}(w)_{s,t}^{1},\cdots ,L_{n}(w)_{s,t}^{n})\text{, \ \ }%
L_{n}(w)_{s,t}^{k}=w_{s,t}^{k}\text{ \ \ \ }\forall (s,t)\in \Delta
\end{equation*}%
$L_{n}(w)^{1}$ is called the first level path of $L_{n}(w)$ which indeed
recovers the original path through $w_{t}=w_{0}+L_{n}(w)_{0,t}^{1}$ up to
the starting point. $L_{n}(w)^{2}$ is called the second level path etc.
Often $L_{n}(w)$ is written as $\boldsymbol{w}$ if no confusion is possible.

$L_{n}(w)$ satisfies an important equation, called Chen's identity 
\begin{equation*}
L_{n}(w)_{s,r}\otimes L_{n}(w)_{r,t}=L_{n}(w)_{s,t}\text{ \ \ \ }\forall
0\leq s<r<t\text{,}
\end{equation*}%
where the tensor product takes place in the truncated tensor algebra $%
T^{(n)}(\mathbb{R}^{d})$. It is indeed the reason that the zeroth term is
taken as $1$ in the definition of $L_{n}(w)$. Chen's identity is nothing but
represents the additivity of iterated integrals over disjoint intervals.

Let $\Omega ^{\infty ,n}(\mathbb{R}^{d})$ denote the totality of all
functions $L_{n}(w)$ where $w$ runs through $\Omega ^{\infty }(\mathbb{R}%
^{d})$: 
\begin{equation*}
\Omega ^{\infty ,n}(\mathbb{R}^{d})=L_{n}\left( \Omega ^{\infty }(\mathbb{R}%
^{d})\right) =\left\{ L_{n}(w):w\in \Omega ^{\infty }(\mathbb{R}%
^{d})\right\} 
\end{equation*}%
which may be naturally identified with the space of all $w\in \Omega
^{\infty }(\mathbb{R}^{d})$ started from $0$ (or any fixed point in $\mathbb{%
R}^{d}$).

Next step is to equip $\Omega ^{\infty ,n}(\mathbb{R}^{d})$ with a metric,
and introduce the concept of geometric rough paths in $\mathbb{R}^{d}$. Let $%
p\geq 1$ be fixed and $[p]$ denote the integer part of $p$, which relates to
the roughness of \emph{sample paths}. The interesting values of $p$ are real
numbers between $2$ and $3$ for the study of Brownian motion in $\mathbb{R}%
^{d}$. The $p$-variation metric, which is the key concept in the analysis of
rough paths, denoted by $d_{p}$, is a metric on $\Omega ^{\infty ,[p]}(%
\mathbb{R}^{d})$ defined by 
\begin{equation*}
d_{p}(L(v),L(w))=\max_{1\leq k\leq \lbrack p]}\sup_{D}\left(
\sum_{l}\left\vert
L(v)_{t_{l-1},t_{l}}^{k}-L(w)_{t_{l-1},t_{l}}^{k}\right\vert ^{\frac{p}{k}%
}\right) ^{\frac{k}{p}}
\end{equation*}%
where $L(w)$ denotes $L_{[p]}(w)$ for simplicity. The completion of $\Omega
^{\infty ,[p]}(\mathbb{R}^{d})$ under $d_{p}$ is denoted by $G\Omega _{p}(%
\mathbb{R}^{d})$. An element in $G\Omega _{p}(\mathbb{R}^{d})$ is called a 
\emph{geometric rough path} in $\mathbb{R}^{d}$ of roughness $p$.

If $\boldsymbol{w}=(1,w^{1},\cdots ,w^{[p]})\in G\Omega _{p}(\mathbb{R}^{d})$%
, then $\boldsymbol{w}$ satisfies Chen's identity $\boldsymbol{w}%
_{s,r}\otimes \boldsymbol{w}_{r,t}=\boldsymbol{w}_{s,t}$ in $T^{[p]}(\mathbb{%
R}^{d})$ (for any $0\leq s<r<t\leq T$), and $\boldsymbol{w}$ has finite $p$%
-variation in the sense that $\sup_{D}\sum_{l}\left\vert
w_{t_{l-1},t_{l}}^{k}\right\vert ^{p/k}<\infty $ \ for all $k\leq \lbrack p]$%
.

T. Lyons \cite{MR1654527} has demonstrated that a theory of integration for
a geometric rough path may be established. Let $\boldsymbol{w}\in G\Omega
_{p}(\mathbb{R}^{d})$ and $f:\mathbb{R}^{d}\rightarrow L(\mathbb{R}^{d},%
\mathbb{R}^{\tilde{d}})$ a function on $\mathbb{R}^{d}$ with values in the
linear space $L(\mathbb{R}^{d},\mathbb{R}^{\tilde{d}})$ of all linear
operators from $\mathbb{R}^{d}$ to $\mathbb{R}^{\tilde{d}}$,\ where $L(%
\mathbb{R}^{d},\mathbb{R}^{\tilde{d}})$ may be identified with $\mathbb{R}%
^{d}\otimes \mathbb{R}^{\tilde{d}}$ or the Euclidean space $\mathbb{R}^{d%
\tilde{d}}$. Such an $f$ is called an $\mathbb{R}^{\tilde{d}}$-valued 1-form
on $\mathbb{R}^{d}$. Let $f^{k}$ denote the $k-1$-th derivative $D^{k-1}f$
of $f$ which is identified with a function on $\mathbb{R}^{d}$ valued in $L((%
\mathbb{R}^{d})^{\otimes k},\mathbb{R}^{\tilde{d}})$ (where $k=1,\cdots ,$).
In particular $f^{1}=f$ .

If $w\in \Omega ^{\infty }(\mathbb{R}^{d})$, then $y=(y_{t})_{t\in \lbrack
0,T]}\in \Omega ^{\infty }(\mathbb{R}^{\tilde{d}})$ where $%
y_{t}=\int_{0}^{t}f(w_{s})dw_{s}$ is the path integral defined via the
Riemann integral 
\begin{equation}
y_{t}=\lim_{m(D)\downarrow 0}\sum_{l}f(w_{t_{l-1}})(w_{t_{l}}-w_{t_{l-1}})%
\text{.}  \label{r-i1}
\end{equation}%
One of the main results in the rough path analysis is the following \emph{%
continuity theorem}.

\begin{theorem}
\label{th1}(T. Lyons \cite{MR1654527}) Suppose that $f\in C_{b}^{n+1}(%
\mathbb{R}^{d};L(\mathbb{R}^{d},\mathbb{R}^{\tilde{d}}))$, where $n\geq 1$
is an integer, and suppose that $p\geq 1$ such that $[p]\leq n$. Then the 
\emph{integration}, which takes $L(w)\in \Omega ^{\infty ,[p]}(\mathbb{R}%
^{d})$ to the lifting $L(y)\in \Omega ^{\infty ,[p]}(\mathbb{R}^{\tilde{d}})$
of  $y$ defined by (\ref{r-i1}), is \emph{continuous} with respect to the $p$%
-variation metrices.\ Moreover, the mapping $L(w)\rightarrow L(y)$ is
uniformly continuous on any bounded set of $\Omega ^{\infty ,[p]}(\mathbb{R}%
^{d})$.
\end{theorem}

If $w\in \Omega ^{\infty }(\mathbb{R}^{d})$ and its lifting to a geometric
rough path $\boldsymbol{w}=(1,w^{1},\cdots ,w^{[p]})$, $y_{t}=%
\int_{0}^{t}f(w_{s})dw_{s}$ the usual Riemannian integral defined as above,
then $y\in \Omega ^{\infty }(\mathbb{R}^{\tilde{d}})$ (which is true
actually for a Lipschitz continuous $f$). The lifting $\boldsymbol{y}%
=L_{[p]}(y)$ is denoted by $\int f(\boldsymbol{w})d\boldsymbol{w}$. The
previous theorem says that the \emph{It\^{o}-Lyons integration} $\boldsymbol{%
w}\rightarrow \int f(\boldsymbol{w})d\boldsymbol{w}$ is continuous with
respect to the $p$-variation metrices. Notice that the usual path integral $%
w\rightarrow \int_{0}^{\cdot }f(w_{s})dw_{s}$ is in general not continuous
under the uniform norm of paths.

As a consequence, for $\boldsymbol{w}\in G\Omega _{p}(\mathbb{R}^{d})$ and $%
f\in C_{b}^{[p]+1}(\mathbb{R}^{d};L(\mathbb{R}^{d},\mathbb{R}^{\tilde{d}}))$
we can definite its integral $\int f(\boldsymbol{w})d\boldsymbol{w}$ as a
unique geometric rough path in $G\Omega _{p}(\mathbb{R}^{\tilde{d}})$. It is
however interesting to know how to define $\int f(\boldsymbol{w})d%
\boldsymbol{w}$ directly by means of rough paths.

Let us describe the definition for $p\in (2,3)$ which is the most
interesting case as it is the case for geometric rough paths associated with
Brownian motion.

Since $[p]=2$, so that we need to define two components $y^{1}$, $y^{2}$
which defines a rough path $\int f(\boldsymbol{w})d\boldsymbol{w}\equiv
(1,y^{1},y^{2})$ where $\boldsymbol{w}=(1,w^{1},w^{2})\in G\Omega _{p}(%
\mathbb{R}^{d})$ and $f\in C_{b}^{3}(\mathbb{R}^{d};L(\mathbb{R}^{d},\mathbb{%
R}^{\tilde{d}}))$. To this end first define $\boldsymbol{\tilde{y}}=(1,%
\tilde{y}^{1},\tilde{y}^{2})$ by 
\begin{equation*}
\tilde{y}%
_{s,t}^{1}=f^{1}(w_{s}^{1})(w_{s,t}^{1})+f^{2}(w_{s}^{1})(w_{s,t}^{2})
\end{equation*}%
and 
\begin{equation*}
\tilde{y}_{s,t}^{2}=f^{1}(w_{s}^{1})\otimes f^{1}(w_{s}^{1})(w_{s,t}^{2})%
\text{.}
\end{equation*}%
It is understandable that $\boldsymbol{\tilde{y}}$ is not a geometric rough
path yet (for example, it does not satisfy Chen's identity in general), so
we take a limiting procedure to define $\boldsymbol{y}=(1,y^{1},y^{2})$ by
means of Riemann sums but at the tensor level. More precisely define 
\begin{equation*}
\boldsymbol{y}_{s,t}=\lim_{m(D)\downarrow 0}\boldsymbol{y}%
_{t_{0},t_{1}}\otimes \cdots \otimes \boldsymbol{y}_{t_{m-1},t_{m}}
\end{equation*}%
where the tensor product $\otimes $ takes place in the truncated algebra $%
T^{2}(\mathbb{R}^{\tilde{d}})$, and the limit $\lim_{m(D)\downarrow 0}$
takes over finite partitions of $[s,t]$. We then can show that the above
limit exists and $\boldsymbol{y}=\int f(\boldsymbol{w})d\boldsymbol{w}$.

\subsection{Differential equations driven by rough paths}

The most important result in the rough path analysis is the universal limit
theorem for solutions of differential equations. Let $\mathbb{R}^{d}$ and $%
\mathbb{R}^{\tilde{d}}$ be two Euclidean spaces. Consider a system of
differential equations of the following form 
\begin{equation}
dy_{t}^{j}=\sum_{i=1}^{d}f_{i}^{j}(y_{t})dw_{t}^{i}\text{ \ \ \ }
\label{de-z1}
\end{equation}%
with initial data $y_{0}\in \mathbb{R}^{\tilde{d}}$, $w=(w^{i})$ is a
continuous path in $\mathbb{R}^{d}$, and $f=(f_{i}^{j}):\mathbb{R}^{\tilde{d}%
}\rightarrow L(\mathbb{R}^{d},\mathbb{R}^{\tilde{d}})$ is a function, called
an $\mathbb{R}^{d}$-valued vector field on $\mathbb{R}^{\tilde{d}}$, where
for $y\in \mathbb{R}^{\tilde{d}}$, $\xi \in \mathbb{R}^{d}$, $(f(y)\xi
)^{j}=\sum_{i=1}^{d}f_{i}^{j}(y)\xi ^{i}$. Suppose $f_{i}^{j}$ are Lipschitz
continuous, and $w\in \Omega ^{\infty }(\mathbb{R}^{d})$, then the standard
Picard iteration applying to the integral equation 
\begin{equation*}
y_{t}^{j}=y_{0}^{j}+\sum_{i=1}^{d}\int_{0}^{t}f_{i}^{j}(y_{s})dw_{s}^{i}
\end{equation*}%
allows to determine a unique continuous path $y$ in $\mathbb{R}^{\tilde{d}}$
with total finite variation. This establishes a mapping sending $w\in \Omega
^{\infty }(\mathbb{R}^{d})$ to the solution $y\in \Omega ^{\infty }(\mathbb{R%
}^{\tilde{d}})$, denoted by $y=F(y_{0},w)$. We then lift both paths $w$ and $%
y$ to their corresponding geometric rough paths of the same roughness $p$,
which thus defines a mapping which maps $L_{[p]}(w)$ to $L_{[p]}(y)$ for
each $p\geq 1$. We will again denote it as $F(y_{0},\cdot )$. That is 
\begin{equation*}
F\left( y_{0},L_{n}(w)\right) =L_{n}\left( F(y_{0},w)\right) \text{ \ \ \ \ }%
\forall w\in \Omega ^{\infty ,n}(\mathbb{R}^{d})\text{.}
\end{equation*}

\begin{theorem}
\label{th2}(T. Lyons \cite{MR1654527})Let $p\geq 1$ and $f=(f_{j}^{i})\in
C_{b}^{[p]+1}(\mathbb{R}^{\tilde{d}};L(\mathbb{R}^{d},\mathbb{R}^{\tilde{d}%
}))$. Then 
\begin{equation*}
F(y_{0},\cdot ):L_{[p]}(w)\rightarrow L_{[p]}\left( F(y_{0},w)\right)
\end{equation*}%
is continuous from $\Omega ^{\infty ,[p]}(\mathbb{R}^{d})$ to $\Omega
^{\infty ,[p]}(\mathbb{R}^{\tilde{d}})$ with respect to the corresponding $p$%
-variation distances.
\end{theorem}

This theorem ensures that there is a unique continuous extension of $F$ on $%
G\Omega _{p}(\mathbb{R}^{d})$, still denoted by $F(y_{0},\cdot )$, so that $%
F(y_{0},\boldsymbol{w})=L_{[p]}\left( F(y_{0},\pi (\boldsymbol{w}))\right) $
if $\boldsymbol{w}\in G\Omega _{p}(\mathbb{R}^{d})$. The continuous mapping 
\begin{equation*}
\boldsymbol{w}\in G\Omega _{p}(\mathbb{R}^{d})\rightarrow F(y_{0},%
\boldsymbol{w})\in G\Omega _{p}(\mathbb{R}^{\tilde{d}})
\end{equation*}%
is called the\emph{\ It\^{o}-Lyons mapping }defined by the differential
equation (\ref{de-z1}).

\section{Main results}

The analysis of rough paths, developed in T. Lyons \cite{MR1654527}, \cite%
{MR2036784}, can be applied to the study of stochastic differential
equations with driven noises which are far more irregular than those of
sample paths of semimartingales. On the other hand, Lyons' theory also sheds
new insight on It\^{o}'s classical theory of stochastic differential
equations, namely, stochastic differential equations driven by Brownian
motion, as presented in \cite{MR637061} for example.

In order to apply Lyons' rough path theory to It\^{o}'s theory of stochastic
differential equations, it is necessary to enhance Brownian motion sample
paths into geometric rough paths. The first construction of Brownian motion
as rough paths was presented in E.-M. Sipilainen's Ph.D. thesis (1993) at
University of Edinburgh, under the supervision of T. Lyons. B. Hambly and T.
Lyons provide further examples in \cite{MR1924401} and \cite{MR1617044}.
They proved that symmetric diffusions with generators of elliptic
differential operators of second order, and Brownian motion on the
Sierpinski gasket can be enhanced into geometric rough paths of level two.
In \cite{MR1883719}, geometric rough paths of level 3 associated with
fractional Brownian motions with Hurst parameter $h>\frac{1}{4}$ were
constructed by means of dyadic approximations, and in \cite{MR2036784}
further examples of geometric rough paths of level 2 or 3 associated a class
of Gaussian processes and more generally a class of stochastic processes
with long time memory are given. In particular, T. Lyons and Z. Qian \cite%
{MR2036784} showed that if the correlation of a continuous stochastic
process over disjoint time intervals satisfies a polynomial decay condition,
together with further less important technical conditions, then the rough
path analysis may be applied to stochastic differential equations driven by
such a process. We notice that the correlation decay condition is a
generalization of the martingale property which is the key in It\^{o}'s
theory of stochastic calculus. We would like to recommend the reader for
other constructions of rough paths to the books by T. Lyons, M. Caruana, and
T. L\'{e}vy \cite{MR2314753}, P. Friz and N. Victoir \cite{MR2604669}, A.
Lejay's survey \cite{MR2053040}, and other papers listed in the references
in \cite{MR2036784} and \cite{MR2604669}.

In particular, the following theorem has been proved (for a proof see for
example \cite{MR2036784}).

\begin{theorem}
\label{th3}If $B$ is a Brownian motion in $\mathbb{R}^{d}$ on a complete
probability space $(\Omega ,\mathcal{F},\mathbb{P})$, $\boldsymbol{B}%
=(1,B^{1},B^{2})$ where $B_{s,t}^{1}=B_{t}-B_{s}$ and 
\begin{equation}
B_{s,t}^{2}=\int_{s<t_{1}<t_{2}<t}\circ dB_{t_{1}}\otimes \circ dB_{t_{2}}
\label{m29-1}
\end{equation}%
for \thinspace $0\leq s<t$, defined in the sense of Stratonovich's
integration, then, for any $2<p<3$%
\begin{equation*}
\mathbb{P}\left\{ \omega :\boldsymbol{B}(\omega )\in G\Omega _{p}(\mathbb{R}%
^{d})\right\} =1\text{.}
\end{equation*}
\end{theorem}

The goal of the present article is to prove a stronger result that $%
\boldsymbol{B}$ are geometric rough paths quasi-surely. To make it more
precise, we need more notions and notations on Malliavin calculus, the
capacity theory on the Wiener space.

Let $\boldsymbol{W}$ denote the space of all continuous paths in $\mathbb{R}%
^{d}$ started at the origin, endowed with the topology of uniform
convergence over finite time intervals. If we wish to emphasize the
dimension $d$ we will render our notations with superscript $d$. $\mathcal{B}%
(\boldsymbol{W})$ denotes the Borel $\sigma $-algebra, which can be
described in terms of the coordinate process $(B_{t})_{t\geq 0}$ on $%
\boldsymbol{W}$, where for each $t\geq 0$, $B_{t}$ is the coordinate
functional on $\boldsymbol{W}$ at time $t$. That is to say, if $w\in 
\boldsymbol{W}$, then $B_{t}(w)=w(t)$. In what follows, in order to avoid
heavy notations, $B_{t}$ may be written as $w(t)$ for $w\in \boldsymbol{W}$, 
$w_{t}(w)$ or $w(t,w)$ if no confusion may arisen. Similar convention
applies to $w$ as well: $w$ may be considered as a typical path in $\mathbb{R%
}^{d}$ or as the canonical coordinate process on $\boldsymbol{W}$. Let $%
\mathcal{F}_{t}^{0}$ to be the smallest $\sigma $-algebra over $\boldsymbol{W%
}$ such that all $B_{s}$ for all $s\leq t$ are $\mathcal{F}_{t}^{0}$%
-measurable. Then $\mathcal{F}^{0}=\sigma \{B_{t}:t\geq 0\}$ coincides with $%
\mathcal{B}(\boldsymbol{W})$.

The Wiener measure $\mathbb{P}$ is the unique probability on $(\boldsymbol{W}%
,\mathcal{B}(\boldsymbol{W}))$ which makes the coordinate process $%
(B_{t})_{t\geq 0}$ a standard Brownian motion. Alternatively, $\mathbb{P}$
is the Gaussian measure on $(\boldsymbol{W},\mathcal{B}(\boldsymbol{W}))$
with the characteristic function%
\begin{equation*}
\int_{\boldsymbol{W}}e^{i\langle l,\cdot \rangle }d\mathbb{P}=e^{-\frac{1}{2}%
|l|_{\boldsymbol{H}}^{2}}\text{ \ \ \ \ for }l\in \boldsymbol{W}^{\ast }
\end{equation*}%
where $\boldsymbol{W}^{\ast }$ is the dual space of $\boldsymbol{W}$,
identified with a linear subspace of the Cameron-Martin space 
\begin{equation*}
\boldsymbol{H}=\{h\in L^{2}(\mathbb{R}_{+},\mathbb{R}^{d})|h(0)=0\text{ and }%
\dot{h}\in L^{2}(\mathbb{R}_{+},\mathbb{R}^{d})\}
\end{equation*}%
equipped with the Hilbert norm $|h|_{\boldsymbol{H}}=\sqrt{\int_{0}^{\infty
}|\dot{h}(t)|^{2}dt}$. It is clear that any $h\in \boldsymbol{H}$ has a
continuous representation and there is a natural continuous imbedding $%
\boldsymbol{H}\hookrightarrow \boldsymbol{W}$.

If $h\in \boldsymbol{H}$ then $\xi _{h}$ denotes the Wiener functional on $%
\boldsymbol{W}$ defined by It\^{o}'s integration $\xi _{h}=\int_{0}^{\infty }%
\dot{h}.dw$. Then, $\xi _{h}$ has a normal distribution $N(0,|h|_{%
\boldsymbol{H}}^{2})$. The translation $\tau _{h}:w\rightarrow w+h$ is
measurable, and $\mathbb{P}\circ \tau _{h}$ is equivalent to the Wiener
measure $\mathbb{P}$, which allows us to define the Malliavin derivative of $%
\xi _{l}$ in a direction $h$:%
\begin{equation*}
D_{h}\xi _{l}=\left. \frac{d}{d\varepsilon }\right\vert _{\varepsilon
=0}\sum_{i=1}^{d}\int_{0}^{\infty }\dot{l}^{i}d(w^{i}+\varepsilon
h^{i})=\langle l,h\rangle _{\boldsymbol{H}}\text{ ,}
\end{equation*}%
Define $D\xi _{l}$ to be $\boldsymbol{H}$-valued random variable on $%
\boldsymbol{W}$ by requiring that $\langle D\xi _{l},h\rangle _{\boldsymbol{H%
}}=D_{h}\xi _{l}$ \ for all $h\in \boldsymbol{H}$ so that $D\xi _{l}=l$ for
any $l\in \boldsymbol{H}$.

If $f$ is a smooth Schwartz function on $\mathbb{R}^{n}$, and $l_{1},\cdots
,l_{n}\in \boldsymbol{H}$, then $F=f(\xi _{l_{1}},\cdots ,\xi _{l_{n}})$ is
called a \textit{smooth Wiener functional} on $\boldsymbol{W}$. The
collection of all such smooth functionals is denoted by $\mathcal{S}$. By
forcing the chain rule to define the Malliavin derivative by 
\begin{equation*}
DF=Df(\xi _{l_{1}},\cdots ,\xi _{l_{n}})=\sum_{j=1}^{n}\frac{\partial f}{%
\partial x^{j}}(\xi _{l_{1}},\cdots ,\xi _{l_{n}})l_{j}\text{.}
\end{equation*}%
By iterating the definition we may define 
\begin{equation*}
D^{2}F=\sum_{j=1}^{n}\frac{\partial ^{2}f}{\partial x^{i}\partial x^{j}}(\xi
_{l_{1}},\cdots ,\xi _{l_{n}})l_{i}\otimes l_{j}
\end{equation*}%
which is an $\boldsymbol{H}^{\otimes 2}$-valued random variable. We can
define $D^{s}F$ inductively, see \cite{MR637061} for more details.

The Sobolev norms $||F||_{q,N}$ (for $N=0,1,2,\cdots $) is equivalent to $%
\sum_{j\leq N}||D^{j}F||_{p}$, where 
\begin{equation*}
||D^{j}F||_{q}=\sqrt[q]{\mathbb{E}\left\{ \left\vert D^{j}F\right\vert _{%
\boldsymbol{H}^{\otimes j}}^{q}\right\} }\text{.}
\end{equation*}%
The completion of smooth Wiener functionals under the norm $||\cdot ||_{q,N}$
is denoted by $\mathbb{D}_{N}^{q}$. For any $q\geq 1$ and integer $N\geq 0$, 
$(\mathbb{D}_{N}^{q},||\cdot ||_{q,N})$ is a Banach space, called a Sobolev
space over the Wiener space $\boldsymbol{W}$. The $(q,N)$-capacity Cap$%
_{q,N} $ is a function on the collection of subsets of $\boldsymbol{W}$
defined as the following.

If $A$ is an open subset of $\boldsymbol{W}$, then%
\begin{equation*}
\text{\textit{Cap}}_{q,N}(A)=\inf \left\{ ||u||_{q,N}:u\in \mathbb{D}_{N}^{q}%
\text{ s.t. }u\geq 1\text{ }\mathbb{P}\text{-a.e. on }A\text{ and }u\geq 0%
\text{ a.s. on }\boldsymbol{W}\right\}
\end{equation*}%
and for a general subset $A\subset \boldsymbol{W}$%
\begin{equation*}
\text{\textit{Cap}}_{q,N}(A)=\inf \left\{ \text{\textit{Cap}}_{q,N}(B):B%
\text{ open and }B\supseteq A\right\} \text{.}
\end{equation*}

For each pair $(q,N)$, \textit{Cap}$_{q,N}$ is a capacity on $\boldsymbol{W}$%
, in the sense that the following hold.

\begin{enumerate}
\item $0\leq $\textit{Cap}$_{q,N}(A)\leq $\textit{Cap}$_{q,N}(B)$ if $%
A\subseteq B$,

\item \textit{Cap}$_{q,N}$ is sub-additive, that is \textit{Cap}$_{q,N}(\cup
_{i=1}^{\infty }A_{i})\leq \sum_{i=1}^{\infty }$\textit{Cap}$_{q,N}(A_{i})$.
\end{enumerate}

If $A$ is Borel measurable, then $\mathbb{P}(A)\leq $\textit{Cap}$%
_{q,N}(A)^{q}$ for any $q>1$ and $N\geq 0$. Therefore any subset of $%
\boldsymbol{W}$ with $(q,N)$-capacity zero is null set with respect to the
Wiener measure $\mathbb{P}$. However there are many probability zero subsets
which have positive $(q,1)$-capacity. According to P. Malliavin, a subset $%
A\subset \boldsymbol{W}$ is called \textit{slim}, if \textit{Cap}$_{q,N}(A)=0
$ for all $q\geq 1$ and $N\in \mathbb{N}$.

We are interested in the random field $\boldsymbol{B}(w)=(1,w^{1},w^{2})$ on 
$\boldsymbol{W}$, valued in $T^{2}(\mathbb{R}^{d})$ parameterized by $%
\{(s,t):s<t\}$, where 
\begin{eqnarray*}
w_{s,t}^{1} &=&w_{t}-w_{s}\text{,} \\
w_{s,t}^{2} &=&\int_{s<t_{1}<t_{2}<t}\circ dw_{t_{1}}\otimes \circ dw_{t_{1}}
\end{eqnarray*}%
for all $s<t$, where $\circ d$ indicates the Stratonovich integration.

We are now in a position to state an interesting consequence to our main
result Theorem \ref{th5}.

\begin{theorem}
\label{th4}For any $p\in (2,3)$ there is a version of $\boldsymbol{B}$ such
that $\left\{ \boldsymbol{B}\notin G\Omega _{p}(\mathbb{R}^{d})\right\} $ is
slim, that is%
\begin{equation*}
\text{Cap}_{q,N}\left\{ w\in \boldsymbol{W}:\boldsymbol{B}(w)\notin G\Omega
_{p}(\mathbb{R}^{d})\right\} =0
\end{equation*}%
for any $q>1$ and $N\in \mathbb{N}$.
\end{theorem}

It is well known and indeed it is very easy to show that there is a version
of the stochastic integrals $B^{2}$ (defined by (\ref{m29-1})) so that $%
\boldsymbol{B}$ is well defined on $\boldsymbol{W}$ except for a slim subset
(i.e. a subset with $(q,N)$-capacity zero for all $q$ and all positive
integer $N$), and $\boldsymbol{B}$ is a rough path quasi-surely, that is, $%
\boldsymbol{B}(w)$ has finite $p$-variation and satisfies Chen's equation
for all $w\in \boldsymbol{W}$ except for a slim subset. Since $2<p<3$, such
a version of $\boldsymbol{B}$ allows us to develop a theory of stochastic
differential equations and thus gives quasi-surely defined solutions for all
stochastic differential equations with coefficients which are regular
enough. However, such a theory will not ensure a convergence theorem such as
Theorem \ref{th6}. In fact, we will prove a quasi-sure approximation theorem
for the Brownian motion.

For a given natural number $n$, $k=0,1,\cdots ,2^{n}$, $t_{n}^{k}=k/2^{n}$
are the dyadic points in $[0,1]$. For a continuous path $w\in \boldsymbol{W}$%
, $w^{(n)}$ is the polygonal approximation defined by 
\begin{equation}
w_{t}^{(n)}=w_{t_{n}^{k-1}}+2^{n}(t-t_{n}^{k-1})(w_{t_{n}^{k}}-w_{t_{n}^{k-1}})%
\text{ \ \ for \ \ }t\in \left[ t_{n}^{k-1},t_{n}^{k}\right] \text{\ }
\label{vz-2}
\end{equation}%
for $k=1,\cdots ,2^{n}$. This notation equally applies to the coordinate
process $\{w_{t}:t\in \lbrack 0,1]\}$. 

The idea of approximating Brownian motion by piece-wise smooth sample paths
originated from the fundamental research of P. L\'{e}vy \cite{levy1}, \cite%
{levy2}, also see K. It\^{o} and H. P. McKean \cite{MR0345224} for the
construction of Brownian motion sample paths starting from polygonal paths
with vortices modelled by random walks.

For simplicity, if no confusion is possible, we will write $\boldsymbol{w}%
^{(m)}$ for $L_{2}(w^{(m)})$, the enhanced geometric rough path of level two
which is associated to $w^{(m)}$. \ Our main result may be stated as the
following

\begin{theorem}
\label{th5}For any $p\in (2,3)$ 
\begin{equation}
\text{\textit{Cap}}_{q,N}\left\{ w\in \boldsymbol{W}:\sum_{m=1}^{\infty
}d_{p}(\boldsymbol{w}^{(m)},\boldsymbol{w}^{(m+1)})=\infty \right\} =0
\label{m28-4}
\end{equation}%
for any $q\geq 1$ and $N\in \mathbb{N}$.
\end{theorem}

It is obvious that (\ref{m28-4}) implies that for any $p$ between $2$ and $3$
there is a subset $A\subset \boldsymbol{W}$ such that \textit{Cap}$%
_{q,N}(A)=0$ for all $q>1$, $N\in \mathbb{N}$, and $\boldsymbol{w}^{(m)}$
converges in $G\Omega _{p}(\mathbb{R}^{d})$ to a limit $\boldsymbol{w}$ on $%
\boldsymbol{W}\setminus A$, which is a modification of $\boldsymbol{B}$.

Putting together with Lyons' universal limit theorem, we obtain immediately
the following quasi-sure limit theorem.

\begin{theorem}
\label{th6}Consider the Stratonovich's type stochastic differential
equations on the Wiener space $(\boldsymbol{W},\mathcal{B}(\boldsymbol{W}),%
\mathbb{P})$ (so that the coordinate process $w=(w_{t})_{t\geq 0}$ is a
standard Brownian motion)%
\begin{equation}
dy_{t}=\sum_{i=1}^{d}f_{i}(y_{t})\circ dw_{t}^{i}+f_{0}(y_{t})dt
\label{m28-5}
\end{equation}%
with initial data $y_{0}$, $f_{i}=(f_{i}^{j})~$($i=0,\cdots ,d$; $j=1,\cdots
,N$). Suppose $f_{i}^{j}$ and $f^{j}$ are in $C_{b}^{3}(\mathbb{R}^{N})$.
Suppose for each $m$, $y^{(m)}$ be the unique solution to the ordinary
equation%
\begin{equation*}
dy_{t}=\sum_{i=1}^{d}f_{i}(y_{t})dw_{t}^{(m),i}+f_{0}(y_{t})dt\text{.}
\end{equation*}%
Then for any $p\in (2,3)$ there is a slim subset $A\subset \boldsymbol{W}$
which is independent of (\ref{m28-5}) such that $\boldsymbol{y}%
^{(m)}=L_{2}(y^{(m)})$ converges to $\boldsymbol{y}=(1,y^{1},y^{2})$ on $%
\boldsymbol{W}\setminus A$, and $y_{t}=y_{0}+y_{0,t}^{1}$ is a version of
the unique strong solution to (\ref{m28-5}).
\end{theorem}

This kind of limit theorems for stochastic differential equations via
ordinary differential equations in the context of almost sure sense have
been discussed by E. McShane \cite{MR0402921}, E. Wong and M. Zakai \cite%
{MR0183023}, D. Stroock and S.R.S. Varadhan \cite{MR0400425}, and etc., see
Section 7, Chapter VI in N. Ikeda and S. Watanabe \cite{MR637061} for a
definite form and for further reference therein. By using Lyons' universal
limit theorem, the Wong-Zakai type limit theorem has been extended to other
rough differential equation driven by symmetric diffusions in \cite%
{MR1617044}, \cite{MR1617044}, by fractional Brownian motions in \cite%
{MR1883719} and by other Gaussian processes in \cite{MR2036784}, \cite%
{MR2604669}, and A.\ M. Davie \cite{MR2387018}, \cite{MR2377011} and etc.

In the context of quasi-sure analysis, partial results (i.e. for a solution
to a single differential equation or special Wiener functionals) have been
obtained by T. Kazumi \cite{MR1145641}, Z. Huang and J. G. Ren \cite%
{MR1082338}, J. G. Ren \cite{MR1056161}, P. Malliavin and D. Nualart \cite%
{MR1222364}, S. Fang \cite{MR1128493}. Our result is a natural
generalization of the preceding mentioned results, and our negligible set is
universal which is independent of the concerned Wiener functionals.

The capacity theory on the infinite dimensional space $\boldsymbol{W}$ was
first studied by P. Malliavin \cite{MR761263}. In fact, Malliavin introduced
the concept of slim sets as negligible sets -- those subsets of $\boldsymbol{%
W}$ with $(q,N)$-capacity zero for all $q>1$ and positive integer $N$, by
using the Ornstein-Uhlenbeck process on the Wiener space $\boldsymbol{W}$.
The current definition of \textit{Cap}$_{q,N}$ was gradually developed
through a series of work by I. Shigekawa \cite{MR1283120}, H. Sugita \cite%
{MR969026}, M.\ Fukushima \cite{MR735979}, \cite{MR729723}, \cite{MR1212113}%
, \cite{MR935979}, \cite{MR723601}, H. Kaneko \cite{MR856891}, M. Takeda 
\cite{MR767798}, J. G. Ren \cite{MR1056156}, F. Hirsch and S. Song \cite%
{MR1296778}, see P. Malliavin \cite{MR1450093} and K. It\^{o} \cite%
{MR1354162} for systematic expositions on slim sets, and M. Fukushima \cite%
{MR569058}, Z.M. Ma and M. R\"{o}ckner \cite{MR1214375} and N. Bouleau and
F. Hirsch \cite{MR1133391} for the capacity theory defined via analytic or
probabilistic potential theory.

Many important almost sure properties which hold for Brownian motion were
proved for the corresponding quasi-sure versions by D. Williams (see the
article by P. Meyer  \cite{meyer1}), M. Fukushima \cite{MR723601}, M. Takeda 
\cite{MR767798}, S. Fang \cite{MR1108611} etc. S. Kusuoka \cite{MR1111196}, 
\cite{MR1151548} initiated the study of non-linear analysis on abstract
Wiener spaces by using capacity theory.

The starting point in our approach is the capacity version of the Borel
Cantelli lemma, that is, if $\sum_{m=1}^{\infty }$\textit{Cap}$%
_{q,s}(A_{m})<\infty $ then \textit{Cap}$_{q,s}(\overline{\lim }%
_{m\rightarrow \infty }A_{m})=0$.

Let $p\in (2,3)$ be fixed, and consider 
\begin{equation*}
A_{m}=\left\{ w\in \boldsymbol{W}:d_{p}(\boldsymbol{w}^{(m+1)},\boldsymbol{w}%
^{(m)})>C\left( \frac{1}{2^{m}}\right) ^{\beta }\right\}
\end{equation*}%
for some $\beta >0$ and constant $C>0$. If we are able to show that 
\begin{equation}
\sum_{m=1}^{\infty }\text{\textit{Cap}}_{q,N}\left( A_{m}\right) <\infty
\label{m29-9}
\end{equation}%
then \textit{Cap}$_{q,N}(\overline{\lim }_{m\rightarrow \infty }A_{m})=0$ so
that 
\begin{equation*}
\text{\textit{Cap}}_{q,N}\left\{ \sum_{m=1}^{\infty }d_{p}(\boldsymbol{w}%
^{(m+1)},\boldsymbol{w}^{(m)})=\infty \right\} \leq \text{\textit{Cap}}%
_{q,N}(\overline{\lim }_{m\rightarrow \infty }A_{m})=0
\end{equation*}%
which yields Theorem \ref{th5}. Therefore, we would like to estimate 
\begin{equation}
\text{\textit{Cap}}_{q,N}\left\{ w\in \boldsymbol{W}:d_{p}(\boldsymbol{w}%
^{(m+1)},\boldsymbol{w}^{(m)})>\lambda \right\} \text{.}  \label{4-22-b1}
\end{equation}

There are few techniques available to bound the capacity such as (\ref%
{4-22-b1}) in contrast to corresponding almost sure statements. In fact the
only effective tool to the knowledge of the present authors is the capacity
maximal inequality (also called the Tchebycheff inequality, see 1.2.5 on
page 92 and 2.2 on page 96 in \cite{MR1450093}), which says that, if $u\in 
\mathbb{D}_{N}^{q}$ and if $u$ is lower semi-continuous or continuous with
respect to the capacity \textit{Cap}$_{q,s}$, then%
\begin{equation}
\text{\textit{Cap}}_{q,N}\left\{ w\in \boldsymbol{W}:u(w)>\lambda \right\}
\leq \frac{C_{q,N}}{\lambda }||u||_{q,N}\text{ \ \ }\forall \lambda >0
\label{m29-5}
\end{equation}%
where $C_{q,N}$ is a constant depending only on $d$, $q$ and $N$. This
requires to estimate the Sobolev norms of $u$. Unfortunately, we are unable
to show (and we do not believe it is true) that $w\rightarrow d_{p}(%
\boldsymbol{w}^{(m+1)},\boldsymbol{w}^{(m)})$ is differentiable in the
Malliavin sense. Instead we consider a metric over paths which dominates the 
$p$-variation distance, but differentiable in Malliavin sense and still good
enough for Brownian motion.

In what follows, $p\in (2,3)$ and $\gamma >\frac{p}{2}-1$ be fixed.

If $\boldsymbol{w}=(1,w^{1},w^{2})$ and $\boldsymbol{\tilde{w}}=(1,\tilde{w}%
^{1},\tilde{w}^{2})$ are two functions on $\Delta $ taking values in $T^{2}(%
\mathbb{R}^{d})$, we consider 
\begin{equation}
\rho _{j}(\boldsymbol{w},\boldsymbol{\tilde{w}})=\left( \sum_{n=1}^{\infty
}n^{\gamma }\sum_{k=1}^{2^{n}}\left\vert w_{t_{n}^{k-1},t_{n}^{k}}^{j}-%
\tilde{w}_{t_{n}^{k-1},t_{n}^{k}}^{j}\right\vert ^{\frac{p}{j}}\right) ^{%
\frac{j}{p}}  \label{d-s1}
\end{equation}%
where $j=1$ or $2$. We will use $\rho _{j}(\boldsymbol{w})$ to denote $\rho
_{j}(\boldsymbol{w},\boldsymbol{\tilde{w}})$ with $\boldsymbol{\tilde{w}}%
=(1,0,0)$. $\rho _{j}$ were invented and used in B. Hambly and T. Lyons \cite%
{MR1617044} for constructing the stochastic area processes associated with
Brownian motions on the Sierpinski gasket. These functionals were used in M.
Ledoux, Z. Qian and T. Zhang \cite{MR1935127} to show the large deviation
principle for Brownian motion under the topology generated by the $p$%
-variation distance. The following estimates have been contained implicitly
in \cite{MR1617044} and made explicit in \cite{MR2036784} and \cite%
{MR1935127}.

\begin{lemma}
\label{le1}Suppose $\gamma >\frac{p}{2}-1$. Then there is a positive
constant $C$ depending only on $\gamma $, $d$ and $p$ such that 
\begin{equation}
\left( \sup_{D}\sum_{l}\left\vert w_{t_{l-1},t_{l}}^{1}\right\vert
^{p}\right) ^{\frac{1}{p}}\leq C\rho _{1}(\boldsymbol{w})\text{,}
\label{es-k1}
\end{equation}%
\begin{equation}
\left( \sup_{D}\sum_{l}\left\vert w_{t_{l-1},t_{l}}^{2}\right\vert ^{\frac{p%
}{2}}\right) ^{\frac{2}{p}}\leq C\left( \rho _{1}(\boldsymbol{w})^{2}+\rho
_{2}(\boldsymbol{w})\right)  \label{es-k2}
\end{equation}%
where $\sup_{D}$ takes over finite partitions $D$ of $[0,1]$, and 
\begin{equation}
d_{p}(\boldsymbol{w},\boldsymbol{\tilde{w}})\leq C\max \left\{ \rho _{1}(%
\boldsymbol{w},\boldsymbol{\tilde{w}}),\rho _{2}(\boldsymbol{w},\boldsymbol{%
\tilde{w}}),\rho _{1}(\boldsymbol{w},\boldsymbol{\tilde{w}})\left( \rho _{1}(%
\boldsymbol{w})+\rho _{1}(\boldsymbol{\tilde{w}})\right) \right\} \text{.}
\label{m29-10}
\end{equation}
\end{lemma}

The idea in our approach is to replace $d_{p}$ by the right-hand side of (%
\ref{m29-10}). Therefore, instead of considering the capacity of $\{d_{p}(%
\boldsymbol{w}^{(m+1)},\boldsymbol{w}^{(m)})>\lambda \}$, we consider the
following%
\begin{equation*}
\rho _{j}(\boldsymbol{w}^{(m+1)},\boldsymbol{w}^{(m)})=\left(
\sum_{n=1}^{\infty }n^{\gamma }\sum_{k=1}^{2^{n}}\left\vert
w_{t_{n}^{k-1},t_{n}^{k}}^{(m+1),j}-w_{t_{n}^{k-1},t_{n}^{k}}^{(m),j}\right%
\vert ^{\frac{p}{j}}\right) ^{\frac{j}{p}}
\end{equation*}%
where $j=1,2$ and $m=1,2,\cdots $. Let%
\begin{equation*}
u_{j}^{(m)}(w)\equiv \rho _{j}(\boldsymbol{w}^{(m+1)},\boldsymbol{w}^{(m)})^{%
\frac{p}{j}}=\sum_{n=1}^{\infty }n^{\gamma }\sum_{k=1}^{2^{n}}\left\vert
w_{t_{n}^{k-1},t_{n}^{k}}^{(m+1),j}-w_{t_{n}^{k-1},t_{n}^{k}}^{(m),j}\right%
\vert ^{\frac{p}{j}}\text{.}
\end{equation*}%
Then, clearly, for each $N=1,2,\cdots ,$%
\begin{equation*}
u_{j}^{(m),N}(w)=\sum_{n=1}^{N}n^{\gamma }\sum_{k=1}^{2^{n}}\left\vert
w_{t_{n}^{k-1},t_{n}^{k}}^{(m+1),j}-w_{t_{n}^{k-1},t_{n}^{k}}^{(m),j}\right%
\vert ^{\frac{p}{j}}
\end{equation*}%
is continuous on $\boldsymbol{W}$ (which is equipped with the uniform norm
over $[0,1]$), and $u_{j}^{(m)}(w)=\sup_{N}u_{j}^{(m),N}(w)$. Therefore $%
u_{j}^{(m)}$ is lower semi-continuous on $\boldsymbol{W}$, and moreover $%
u_{j}^{(m)}\in \mathbb{D}_{1}^{q}$ for any $q\geq 1$. Therefore we may apply
(\ref{m29-5}) to deduce that%
\begin{equation}
\text{\textit{Cap}}_{q,1}\left\{ \rho _{j}(\boldsymbol{w}^{(m+1)},%
\boldsymbol{w}^{(m)})^{\frac{p}{j}}>\lambda \right\} \leq \frac{C_{q}}{%
\lambda }\left\Vert \rho _{j}(\boldsymbol{w}^{(m+1)},\boldsymbol{w}^{(m)})^{%
\frac{p}{j}}\right\Vert _{q,1}  \label{m29-6}
\end{equation}%
and similarly%
\begin{eqnarray}
&&\text{\textit{Cap}}_{q,1}\left\{ \rho _{1}(\boldsymbol{w}^{(m)})^{p}\rho
_{1}(\boldsymbol{w}^{(m+1)},\boldsymbol{w}^{(m)})^{p}>\lambda \right\}  
\notag \\
&\leq &\frac{C_{q}}{\lambda }\left\Vert \rho _{1}(\boldsymbol{w}%
^{(m)})^{p}\rho _{1}(\boldsymbol{w}^{(m+1)},\boldsymbol{w}%
^{(m)})^{p}\right\Vert _{q,1}  \label{m29-7}
\end{eqnarray}%
for any $\lambda >0$, where $C_{q}>0$ depends only on $q$ and $d$. \ In the
next section we will establish the necessary estimates to ensure (\ref{m29-9}%
) for the case that $N=1$.

This approach can not be extended to $(q,N)$-capacity case with $N\geq 2$,
this is because of a simple reason that our dominated distance $\rho _{j}(%
\boldsymbol{w}^{(m+1)},\boldsymbol{w}^{(m)})^{\frac{p}{j}}$ does not belong
to $\mathbb{D}_{N}^{q}$ for $N\geq 2$. We need new idea to estimate the $%
(q,N)$-capacity for $N\geq 2$. The observation to get around this difficulty
is that the capacity of $\{\rho _{j}(\boldsymbol{w}^{(m+1)},\boldsymbol{w}%
^{(m)})^{p/j}>\lambda \}$ is \textquotedblleft evenly\textquotedblright\
distributed over the dyadic partitions, which allows to reduce our task to
estimating the capacities of some polynomials of Brownian motion sample
paths, which will be explained in the last sectionof the article, where we
conclude the proof of Theorem \ref{th5}.

\section{Some technical estimates}

In this section we establish several technical estimates which will be used
in the construction quasi-surely defined geometric rough paths associated
with Brownian motion.

We are going to use the following notations. If $J\subset \lbrack 0,\infty )$
is a finite interval, then $1_{J}$ is the characteristic function of $J$,
and $\boldsymbol{1}_{J}\in \boldsymbol{H}$, which is $\mathbb{R}^{d}$-valued
function with the same component $\int_{0}^{\cdot }1_{J}(s)ds$. Hence $|%
\boldsymbol{1}_{J}|_{\boldsymbol{H}}=\sqrt{d}\sqrt{|J|}$ where $|J|$ denotes
the length of the interval $J$. We will frequently use the following
elementary fact: if $\{J_{i}:i=1,\cdots ,n\}$ is a family of \textit{disjoint%
} finite intervals, then%
\begin{equation}
\left\vert \sum_{i=1}^{n}\boldsymbol{1}_{J_{i}}\right\vert _{\boldsymbol{H}}=%
\sqrt{d}\sqrt{\sum_{i=1}^{n}|J_{i}|}\text{ .}  \label{4-17-1}
\end{equation}%
The corresponding fact for the increments of Brownian motion instead of
characteristic functions is the context of the following lemma.

\begin{lemma}
\label{le2}Then there is a constant $C>0$ depending only on $d$, such that%
\begin{equation}
\left\Vert \sum_{i=1}^{N}\xi _{i}\otimes \tilde{\xi}_{i}\right\Vert _{q}\leq
Cq\sqrt{N}  \label{t-t-5}
\end{equation}%
for any $\xi _{i}$, $\tilde{\xi}_{j}$ which are independent valued in $%
\mathbb{R}^{d}$, with the standard normal $N(0,1_{\mathbb{R}^{d}})$, for any 
$q\geq 1$ and $N\in \mathbb{N}$.
\end{lemma}

This lemma follows from a simple application of the hypercontractivity of
the O-U semigroup.

The increment over the interval $J_{n}^{k}\equiv (t_{n}^{k-1},t_{n}^{k}]$ of 
$w\in \boldsymbol{W}$ is denoted by $\xi _{n}^{k}(w)$ or simply by $\xi
_{n}^{k}$ (which denotes the functional $w\rightarrow \xi _{n}^{k}(w)$ as
well by abusing the notation), if no confusion may arise. That is%
\begin{equation}
\xi _{n}^{k}=w_{\frac{k}{2^{n}}}-w_{\frac{k-1}{2^{n}}}\text{, \ \ }%
k=1,\cdots ,2^{n}\text{.}  \label{4-17-3}
\end{equation}

For each $n$, since $\{J_{n}^{k}:k=1,\cdots ,2^{n}\}$ are disjoint, $\{$ $%
\xi _{n}^{k}:k=1,\cdots ,2^{n}\}$ are independent, identically distributed
with normal distribution $N(0,2^{-n}I_{\mathbb{R}^{d}})$.

Recall that $\boldsymbol{w}^{(n)}=L_{2}(w^{(n)})$ with first level component 
$w^{(n),1}$ and second level $w^{(n),2}$ respectively, so that 
\begin{equation}
\left\{ 
\begin{array}{ccc}
w_{s,t}^{(n),1} & = & w_{t}^{(n)}-w_{s}^{(n)}\text{,} \\ 
w_{s,t}^{(n),2} & = & \int_{s<t_{1}<t_{2}<t}dw_{t_{1}}^{(n)}\otimes
dw_{t_{2}}^{(n)}\text{.}%
\end{array}%
\right.  \label{m21-2}
\end{equation}

It is easy to see that 
\begin{equation}
w_{t_{n}^{k-1},t_{n}^{k}}^{(m),1}=\left\{ 
\begin{array}{ccc}
\xi _{n}^{k}\text{ \ \ \ \ \ \ \ \ \ } & \text{for }n<m\text{,\ } &  \\ 
\frac{2^{m}}{2^{n}}\xi _{m}^{k(n,m)} & \text{for }n\geq m\text{ \ } & 
\end{array}%
\right. \text{\ \ \ }  \label{m-02}
\end{equation}%
where $k=1,\cdots ,2^{n}$, and in the case $n>m$, $k(n,m)$ is the unique
integer $l$ between $1$ and $2^{m}$ such that 
\begin{equation}
t_{m}^{l-1}\leq t_{n}^{k-1}<t_{n}^{k}<t_{m}^{l}\text{.}  \label{m-03}
\end{equation}

In order to write down some formulas which will be used in what follows, it
is better to use Possion bracket operations $[,]$ and $\{,\}$, that is, if $%
\xi ,\eta \in \mathbb{R}^{d}$, then%
\begin{equation}
\lbrack \xi ,\eta ]=\xi \otimes \eta -\eta \otimes \xi  \label{4-17-4}
\end{equation}%
and%
\begin{equation}
\{\xi ,\eta \}=\xi \otimes \eta +\eta \otimes \xi \text{,}  \label{4-17-5}
\end{equation}%
while we will reserve the sharp bracket $\langle a,b\rangle $ to denote the
scalar product in the Euclidean spaces, or in Hilbert space $\boldsymbol{H}%
^{\otimes k}$.

With these notations, we have (see Section 4.2 in \cite{MR2036784} for
details)

\begin{equation}
w_{t_{n}^{k-1},t_{n}^{k}}^{(m),2}=\frac{1}{2}\xi _{n}^{k}\otimes \xi
_{n}^{k}+\frac{1}{2}\sum_{\substack{ r,s=2^{m-n}(k-1)+1  \\ r<s}}%
^{2^{m-n}k}[\xi _{m}^{r},\xi _{m}^{s}]  \label{m19.01}
\end{equation}%
for $n<m$, so that, if $n<m$, then 
\begin{equation}
w_{t_{n}^{k-1},t_{n}^{k}}^{(m+1),2}-w_{t_{n}^{k-1},t_{n}^{k}}^{(m),2}=\frac{1%
}{2}\sum_{r=2^{m-n}(k-1)+1}^{2^{m-n}k}[\xi _{m+1}^{2r-1},\xi _{m+1}^{2r}]%
\text{.}  \label{i-as1}
\end{equation}%
If $n\geq m$, then 
\begin{equation}
w_{t_{n}^{k-1},t_{n}^{k}}^{(m),2}=\frac{1}{2}\frac{2^{2m}}{2^{2n}}\xi
_{m}^{k(n,m)}\otimes \xi _{m}^{k(n,m)}\text{.}  \label{m19.02}
\end{equation}

\subsection{Moment estimates under Sobolev norms}

In this part, let $p\in (2,3)$ is a constant, $d$ is the dimension, $n,m\in 
\mathbb{N}$. We wish to develop several moment estimates for $%
w_{t_{n}^{k-1},t_{n}^{k}}^{(m+1),j}-w_{t_{n}^{k-1},t_{n}^{k}}^{(m),j}$ (for $%
j=1,2$). \ 

\begin{lemma}
\label{lem1a}There is a constant $C$ depending only on $d$ such that%
\begin{equation}
\left\Vert w_{t_{n}^{k-1},t_{n}^{k}}^{(m),j}\right\Vert _{q}\leq \left\{ 
\begin{array}{cc}
C\left( \sqrt{q}\frac{1}{\sqrt{2^{n}}}\right) ^{j} & \text{ \ for }n<m\text{
,} \\ 
C\left( \sqrt{q}\frac{\sqrt{2^{m}}}{2^{n}}\right) ^{j} & \text{\ for }n\geq m%
\end{array}%
\right.  \label{4-17-6}
\end{equation}%
where $j=1,2$, 
\begin{equation}
\left\Vert
w_{t_{n}^{k-1},t_{n}^{k}}^{(m+1),1}-w_{t_{n}^{k-1},t_{n}^{k}}^{(m),1}\right%
\Vert _{q}\leq \left\{ 
\begin{array}{cc}
0\text{, \ \ \ \ \ \ \ } & \text{ \ if }n\leq m\text{, \ } \\ 
C\sqrt{q}\sqrt{\frac{2^{m}}{2^{2n}}} & \text{ \ if }n\geq m\text{,}%
\end{array}%
\right.  \label{m22-4}
\end{equation}%
and%
\begin{equation}
\left\Vert
w_{t_{n}^{k-1},t_{n}^{k}}^{(m+1),2}-w_{t_{n}^{k-1},t_{n}^{k}}^{(m),2}\right%
\Vert _{q}\leq \left\{ 
\begin{array}{cc}
Cq\sqrt{\frac{1}{2^{m+n}}} & \text{ \ if }n\leq m\text{, \ } \\ 
Cq\frac{2^{m}}{2^{2n}}\text{ \ \ \ \ } & \text{ \ if }n\geq m\text{.}%
\end{array}%
\right.  \label{m22-05}
\end{equation}%
for any $q\geq 1$.
\end{lemma}

\begin{proof}
For simplicity, let $Y_{j}=w_{t_{n}^{k-1},t_{n}^{k}}^{(m),j}$ and $%
X_{j}=w_{t_{n}^{k-1},t_{n}^{k}}^{(m+1),j}-w_{t_{n}^{k-1},t_{n}^{k}}^{(m),j}$%
. For $j=1$, (\ref{4-17-6}) follows from the fact that%
\begin{equation}
\left\Vert Y_{1}\right\Vert _{q}=\left\{ 
\begin{array}{cc}
\frac{1}{\sqrt{2^{n}}}||\xi ||_{q} & \text{ \ for }n\leq m\text{ ,} \\ 
\frac{\sqrt{2^{m}}}{2^{n}}||\xi ||_{q} & \text{\ for }n>m\,%
\end{array}%
\right.  \label{m27-4}
\end{equation}%
where $\xi \sim N(0,1_{\mathbb{R}^{d}})$, and the fact that $||\xi
||_{q}\leq C\sqrt{q}$ for some constant $C$ depending only on $d$. For $j=2$
and $n>m$, (\ref{4-17-6}) follows from (\ref{m19.02}) directly. Consider the
case that $n<m$. According to (\ref{m19.01}), we need to estimate%
\begin{equation*}
I_{q}=\left\Vert \sum_{\substack{ r,s=2^{m-n}(k-1)+1  \\ r<s}}^{2^{m-n}k}\xi
_{m}^{r}\otimes \xi _{m}^{s}\right\Vert _{q}=\left\Vert
\sum_{s=2^{m-n}(k-1)+1}^{2^{m-n}k}\sum_{r=2^{m-n}(k-1)+1}^{s-1}\xi
_{m}^{r}\otimes \xi _{m}^{s}\right\Vert _{q}\text{.}
\end{equation*}%
Since $\xi _{m}^{r}\otimes \xi _{m}^{s}$ belong to the second chaos
component for whatever $m$, so that $I_{q}\leq C_{1}qI_{2}$ for some
constants $C_{1}$ and $C_{2}$ depending only on $d$, but independent of $m$
or $n$. Therefore we may assume that $q=2$. Furthermore, for simplicity, set 
$\eta _{s}=\sum_{r=2^{m-n}(k-1)+1}^{s-1}\xi _{m}^{r}$. Then $\xi _{s}$ has a
normal distribution with mean zero and 
\begin{equation*}
\text{var}(\eta _{s})=\frac{1}{2^{m}}(s-2^{m-n}(k-1))\text{.}
\end{equation*}%
Thus%
\begin{eqnarray*}
\left\Vert \sum_{s=2^{m-n}(k-1)+1}^{2^{m-n}k}\eta _{s}\otimes \xi
_{m}^{s}\right\Vert _{2}^{2} &=&\sum_{i,j=1}^{d}\mathbb{E}\left(
\sum_{s=2^{m-n}(k-1)+1}^{2^{m-n}k}\eta _{s}^{i}\xi _{m}^{s,j}\right) ^{2} \\
&=&\sum_{i,j=1}^{d}\sum_{s=2^{m-n}(k-1)+1}^{2^{m-n}k}\mathbb{E}\left( \eta
_{s}^{i}\xi _{m}^{s,j}\right) ^{2} \\
&&+2\sum_{i,j=1}^{d}\sum_{\substack{ s=2^{m-n}(k-1)+1  \\ s<r}}^{2^{m-n}k}%
\mathbb{E}\left( \eta _{s}^{i}\xi _{m}^{s,j}\eta _{r}^{i}\xi
_{m}^{r,j}\right) \text{.}
\end{eqnarray*}%
The last term has contribution zero. This is because for $r>s$, $\xi
_{m}^{r,j}$ is independent of $\eta _{s}^{i}\xi _{m}^{s,j}\eta _{r}^{i}$, so
that 
\begin{equation*}
\mathbb{E}\left( \eta _{s}^{i}\xi _{m}^{s,j}\eta _{r}^{i}\xi
_{m}^{r,j}\right) =\mathbb{E}\left( \eta _{s}^{i}\xi _{m}^{s,j}\eta
_{r}^{i}\right) \mathbb{E}\xi _{m}^{r,j}=0\text{ \ for }s<r\text{.}
\end{equation*}%
Hence 
\begin{eqnarray*}
\left\Vert \sum_{s=2^{m-n}(k-1)+1}^{2^{m-n}k}\eta _{s}\otimes \xi
_{m}^{s}\right\Vert _{2}^{2}
&=&\sum_{i,j=1}^{d}\sum_{s=2^{m-n}(k-1)+1}^{2^{m-n}k}\mathbb{E}\left( \eta
_{s}^{i}\xi _{m}^{s,j}\right) ^{2} \\
&=&\sum_{i,j=1}^{d}\sum_{s=2^{m-n}(k-1)+1}^{2^{m-n}k}\mathbb{E}\left( \xi
_{m}^{s,j}\right) ^{2}\mathbb{E}\left( \eta _{s}^{i}\right) ^{2} \\
&=&d^{2}\sum_{s=2^{m-n}(k-1)+1}^{2^{m-n}k}\frac{1}{2^{m}}(s-2^{m-n}(k-1))%
\frac{1}{2^{m}} \\
&=&\frac{1}{2^{2m}}d^{2}\sum_{s=1}^{2^{m-n}}s\leq C\frac{1}{2^{2n}}
\end{eqnarray*}%
and therefore%
\begin{equation*}
\left\Vert \sum_{\substack{ r,s=2^{m-n}(k-1)+1  \\ r<s}}^{2^{m-n}k}\xi
_{m}^{r}\otimes \xi _{m}^{s}\right\Vert _{q}\leq C_{2}q\frac{1}{2^{n}}
\end{equation*}%
for some constant $C$ depending only on $d$. By (\ref{m21-2}) \ and the
preceding estimate we have%
\begin{eqnarray*}
\left\Vert Y_{2}\right\Vert _{q} &\leq &\frac{1}{2}\left\Vert \xi
_{n}^{k}\otimes \xi _{n}^{k}\right\Vert _{q}+\frac{1}{2}\left\Vert \sum 
_{\substack{ r,s=2^{m-n}(k-1)+1  \\ r<s}}^{2^{m-n}k}[\xi _{m}^{r},\xi
_{m}^{s}]\right\Vert _{q} \\
&\leq &\frac{1}{2}\left\Vert \xi _{n}^{k}\otimes \xi _{n}^{k}\right\Vert
_{q}+\left\Vert \sum_{\substack{ r,s=2^{m-n}(k-1)+1  \\ r<s}}^{2^{m-n}k}\xi
_{m}^{r}\otimes \xi _{m}^{s}\right\Vert _{q} \\
&\leq &C\frac{q}{2^{n}}\text{.}
\end{eqnarray*}%
Remain to show (\ref{m22-05}) for the case $n\leq m$. Indeed, by (\ref{i-as1}%
) and (\ref{t-t-5})%
\begin{eqnarray*}
\left\Vert
w_{t_{n}^{k-1},t_{n}^{k}}^{(m+1),2}-w_{t_{n}^{k-1},t_{n}^{k}}^{(m),2}\right%
\Vert _{q} &=&\frac{1}{2}\left\Vert \sum_{r=2^{m-n}(k-1)+1}^{2^{m-n}k}[\xi
_{m+1}^{2r-1},\xi _{m+1}^{2r}]\right\Vert _{q} \\
&\leq &Cq\sqrt{\frac{1}{2^{m+n}}}
\end{eqnarray*}%
which completes the proof of the lemma.
\end{proof}

\begin{lemma}
\label{le6}There is a constant $C$ depending only on $d$ such that If $n\geq
m$, then $Y_{1}=\frac{2^{m}}{2^{n}}\xi _{m}^{k(n,m)}$, $DY_{1}=\frac{2^{m}}{%
2^{n}}\boldsymbol{1}_{J_{m}^{k(n,m)}}$If $n<m$, then $Y_{1}=\xi _{n}^{k}$ so
that $DY_{1}=\boldsymbol{1}_{J_{n}^{k}}$%
\begin{equation}
\left\vert Dw_{t_{n}^{k-1},t_{n}^{k}}^{(m),1}\right\vert _{\boldsymbol{H}%
}\leq \left\{ 
\begin{array}{cc}
C\sqrt{\frac{1}{2^{n}}} & \text{if }n\leq m\text{,} \\ 
C\frac{2^{m}}{2^{n}}\sqrt{\frac{1}{2^{m}}} & \text{if }n>m\text{,}%
\end{array}%
\right.  \label{4-21-1}
\end{equation}
\begin{equation}
\left\vert D\left(
w_{t_{n}^{k-1},t_{n}^{k}}^{(m+1),1}-w_{t_{n}^{k-1},t_{n}^{k}}^{(m),1}\right)
\right\vert _{\boldsymbol{H}}\leq \left\{ 
\begin{array}{cc}
0\text{ \ \ \ \ \ \ \ \ \ \ \ \ \ \ \ \ } & \text{if }n\leq m\text{,} \\ 
C\sqrt{\frac{2^{m}}{2^{n}}}\sqrt{\frac{1}{2^{n}}} & \text{if }n>m%
\end{array}%
\right.  \label{m23-1}
\end{equation}%
and 
\begin{equation}
\left\Vert D^{a}\left(
w_{t_{n}^{k-1},t_{n}^{k}}^{(m+1),2}-w_{t_{n}^{k-1},t_{n}^{k}}^{(m),2}\right)
\right\Vert _{q}\leq \left\{ 
\begin{array}{cc}
C\sqrt{q}^{2-a}\sqrt{\frac{1}{2^{m+n}}}\text{ \ \ \ \ \ \ } & \text{\ if }%
n\leq m\text{,} \\ 
C\sqrt{q}^{2-a}\sqrt{\frac{2^{3m}}{2^{3n}}}\sqrt{\frac{1}{2^{n+m}}} & \text{%
if }n>m%
\end{array}%
\right. \text{ \ }  \label{m22-6}
\end{equation}%
for any $m,n\in \mathbb{N}$, and $k=1,\cdots ,2^{n}$, $a=1,2$.
\end{lemma}

\begin{proof}
(\ref{4-21-1}) is easy to see. Let $%
X_{j}=w_{t_{n}^{k-1},t_{n}^{k}}^{(m+1),j}-w_{t_{n}^{k-1},t_{n}^{k}}^{(m),j}$
for simplicity. If $n\leq m$ then $X_{1}=0$, and if $n\geq m$, then%
\begin{equation*}
DX_{1}=\frac{2^{m+1}}{2^{n}}\boldsymbol{1}_{J_{m+1}^{k(n,m+1)}}-\frac{2^{m}}{%
2^{n}}\boldsymbol{1}_{J_{m}^{k(n,m)}}
\end{equation*}%
so that%
\begin{equation*}
\left\vert DX_{1}\right\vert _{\boldsymbol{H}}=C\sqrt{\frac{2^{m}}{2^{n}}}%
\sqrt{\frac{1}{2^{n}}}
\end{equation*}%
where $C$ is a constant depending only on $d$. Next we consider the L\'{e}vy
area $X_{2}$. If $n<m$ 
\begin{equation*}
X_{2}=\frac{1}{2}\sum_{l=2^{m-n}(k-1)+1}^{2^{m-n}k}[\xi _{m+1}^{2l-1},\xi
_{m+1}^{2l}]\text{.}
\end{equation*}%
so that%
\begin{equation*}
DX_{2}=\frac{1}{2}\sum_{l=2^{m-n}(k-1)+1}^{2^{m-n}k}\left( [\boldsymbol{1}%
_{J_{m+1}^{2l-1}},\xi _{m+1}^{2l}]+[\xi _{m+1}^{2l-1},\boldsymbol{1}%
_{J_{m+1}^{2l}}]\right) \text{ }
\end{equation*}%
and%
\begin{equation*}
D^{2}X_{2}=\frac{1}{2}\sum_{l=2^{m-n}(k-1)+1}^{2^{m-n}k}\left( [\boldsymbol{1%
}_{J_{m+1}^{2l-1}},\boldsymbol{1}_{J_{m+1}^{2l}}]-[\boldsymbol{1}%
_{J_{m+1}^{2l}},\boldsymbol{1}_{J_{m+1}^{2l-1}}]\right)
\end{equation*}%
where 
\begin{equation*}
\lbrack \boldsymbol{1}_{J_{m+1}^{2l-1}},\boldsymbol{1}%
_{J_{m+1}^{2l}}](t_{1},t_{2})=[\boldsymbol{1}_{J_{m+1}^{2l-1}}(t_{1}),%
\boldsymbol{1}_{J_{m+1}^{2l}}(t_{2})]\text{.}
\end{equation*}%
Since the intervals $J_{m+1}^{i}$ are disjoint, so that%
\begin{eqnarray*}
|DX_{2}|_{\boldsymbol{H}}^{2} &=&\int_{0}^{\infty }\frac{1}{4}%
\sum_{i,j}\left( \sum_{l=2^{m-n}(k-1)+1}^{2^{m-n}k}\left(
[1_{J_{m+1}^{2l-1}},\xi _{m+1}^{2l}]^{ij}+[\xi
_{m+1}^{2l-1},1_{J_{m+1}^{2l}}]^{ij}\right) \right) ^{2} \\
&=&\int_{0}^{\infty }\frac{1}{4}\sum_{i,j}\sum_{l=2^{m-n}(k-1)+1}^{2^{m-n}k}%
\left( [1_{J_{m+1}^{2l-1}},\xi _{m+1}^{2l}]^{ij}+[\xi
_{m+1}^{2l-1},1_{J_{m+1}^{2l}}]^{ij}\right) ^{2} \\
&\leq &\frac{1}{2}\sum_{l=2^{m-n}(k-1)+1}^{2^{m-n}k}\int_{0}^{\infty }\left(
|[1_{J_{m+1}^{2l-1}},\xi _{m+1}^{2l}]|^{2}+|[\xi
_{m+1}^{2l-1},1_{J_{m+1}^{2l}}]|^{2}\right) \\
&\leq &\frac{1}{2^{m+1}}\sum_{l=2^{m-n}(k-1)+1}^{2^{m-n}k}\left( |\xi
_{m+1}^{2l}|^{2}+|\xi _{m+1}^{2l-1}|^{2}\right)
\end{eqnarray*}%
that is%
\begin{equation}
|DX_{2}|_{\boldsymbol{H}}\leq \sqrt{\frac{1}{2^{m+1}}}\sqrt{%
\sum_{l=2^{m-n}(k-1)+1}^{2^{m-n}k}\left( |\xi _{m+1}^{2l}|^{2}+|\xi
_{m+1}^{2l-1}|^{2}\right) }\text{. }  \label{4-20-1}
\end{equation}%
Similarly,%
\begin{eqnarray*}
|D^{2}X_{2}|_{\boldsymbol{H}^{\otimes 2}}^{2} &=&\int_{0}^{\infty
}\int_{0}^{\infty }\frac{1}{4}\sum_{i,j}\left(
\sum_{l=2^{m-n}(k-1)+1}^{2^{m-n}k}\left(
[1_{J_{m+1}^{2l-1}},1_{J_{m+1}^{2l}}]-[1_{J_{m+1}^{2l}},1_{J_{m+1}^{2l-1}}]%
\right) ^{ij}\right) ^{2} \\
&=&\int_{0}^{\infty }\int_{0}^{\infty }\frac{1}{4}%
\sum_{l=2^{m-n}(k-1)+1}^{2^{m-n}k}%
\sum_{i,j}|[1_{J_{m+1}^{2l-1}},1_{J_{m+1}^{2l}}]-[1_{J_{m+1}^{2l}},1_{J_{m+1}^{2l-1}}]|^{2}
\\
&\leq &\int_{0}^{\infty }\int_{0}^{\infty
}\sum_{l=2^{m-n}(k-1)+1}^{2^{m-n}k}|[1_{J_{m+1}^{2l-1}},1_{J_{m+1}^{2l}}]| \\
&=&C2^{m-n}\frac{1}{2^{2(m+1)}}
\end{eqnarray*}%
where $C$ depends only on $d$, so that 
\begin{equation}
|D^{2}X_{2}|_{\boldsymbol{H}^{\otimes 2}}\leq C\sqrt{\frac{1}{2^{m+n}}}\text{
.}  \label{4-20-2}
\end{equation}%
Hence, for $q\geq 2$ we have%
\begin{eqnarray*}
\left\Vert DX_{2}\right\Vert _{q} &\leq &\sqrt{\frac{1}{2^{m+1}}}\sqrt{%
\left\Vert \sum_{l=2^{m-n}(k-1)+1}^{2^{m-n}k}\left( |\xi
_{m+1}^{2l}|^{2}+|\xi _{m+1}^{2l-1}|^{2}\right) \right\Vert _{\frac{q}{2}}}
\\
&\leq &\text{ }\sqrt{\frac{1}{2^{m+1}}}\sqrt{%
\sum_{l=2^{m-n}(k-1)+1}^{2^{m-n}k}\left( ||\xi _{m+1}^{2l}||_{q}^{2}+||\xi
_{m+1}^{2l-1}||_{q}^{2}\right) } \\
&\leq &\text{ }C\sqrt{q}\sqrt{\frac{1}{2^{m+1}}}\sqrt{2^{m-n}\frac{1}{2^{m}}}%
\leq C\sqrt{q}\sqrt{\frac{1}{2^{m+n}}}
\end{eqnarray*}%
and%
\begin{equation}
||D^{2}X_{2}||_{q}\leq C\frac{1}{\sqrt{2^{m+n}}}\text{ \ \ \ \ }\forall
n\leq m\text{.}  \label{4-20-3}
\end{equation}

If $n\geq m$ then%
\begin{eqnarray*}
DX_{2} &=&\frac{1}{2}\frac{2^{2m}}{2^{2n}}\left( 4\{\xi _{m+1}^{k(n,m+1)},%
\boldsymbol{1}_{J_{m+1}^{k(n,m+1)}}\}-\{\xi _{m}^{k(n,m)},\boldsymbol{1}%
_{J_{m}^{k(n,m)}}\}\right) \text{,} \\
D^{2}X_{2} &=&\frac{1}{2}\frac{2^{2m}}{2^{2n}}\left( 4\{\boldsymbol{1}%
_{J_{m+1}^{k(n,m+1)}},\boldsymbol{1}_{J_{m+1}^{k(n,m+1)}}\}-\{\boldsymbol{1}%
_{J_{m}^{k(n,m)}},\boldsymbol{1}_{J_{m}^{k(n,m)}}\}\right)
\end{eqnarray*}%
so that%
\begin{equation*}
|DX_{2}|_{\boldsymbol{H}}\leq C\sqrt{\frac{2^{3m}}{2^{3n}}}\sqrt{\frac{1}{%
2^{n}}}\left( |\xi _{m+1}^{k(n,m+1)}|+|\xi _{m}^{k(n,m)}|\right) \text{, \ }%
|D^{2}X_{2}|_{\boldsymbol{H}}\leq C\frac{2^{2m}}{2^{2n}}\frac{1}{2^{m}}
\end{equation*}%
for a constant $C$ depending only on $d$. After taking $q$-norm, we obtain (%
\ref{m22-6}).
\end{proof}

\begin{lemma}
\label{le7}Let $\tilde{N}>0$. Consider the following functionals $%
f_{m,n,k}^{j}$ on the Wiener space $\boldsymbol{W}$ defined by%
\begin{equation}
f_{m,n,k}^{j}(w)=\left\vert
w_{t_{n}^{k-1},t_{n}^{k}}^{(m+1),j}-w_{t_{n}^{k-1},t_{n}^{k}}^{(m),j}\right%
\vert ^{2\tilde{N}}\text{ \ \ \ \ for }w\in \boldsymbol{W}\text{.}
\label{4-21-3}
\end{equation}%
If $2\tilde{N}>1$ then there is a constant $C$ depending only on $d$ and $%
\tilde{N}$ such that%
\begin{equation}
\left\Vert Df_{m,n,k}^{1}\right\Vert _{q}\leq Cq^{\tilde{N}}\left( \frac{%
2^{m}}{2^{2n}}\right) ^{\tilde{N}}\text{ \ \ for }n>m  \label{m23-6}
\end{equation}%
and 
\begin{equation}
\left\Vert Df_{m,n,k}^{2}\right\Vert _{q}\leq \left\{ 
\begin{array}{cc}
C\sqrt{q}^{4\tilde{N}}\left( \frac{1}{2^{m+n}}\right) ^{\tilde{N}} & \text{\
if }n\leq m\text{,} \\ 
C\sqrt{q}^{4\tilde{N}}\left( \frac{2^{m}}{2^{2n}}\text{ }\right) ^{2\tilde{N}%
} & \text{if }n>m\text{ .}%
\end{array}%
\right.  \label{m23-5}
\end{equation}%
for any $n,m$ and $k=1,\cdots ,2^{n}$.
\end{lemma}

\begin{proof}
Let $%
X_{j}=w_{t_{n}^{k-1},t_{n}^{k}}^{(m+1),j}-w_{t_{n}^{k-1},t_{n}^{k}}^{(m),j}$
\ ($j=1,2$) as in the proof of the previous lemma. Since $2\tilde{N}>1$ so
that $f_{m,n,k}^{j}$ are differentiable in the sense of Malliavin calculus,
and by chain rule 
\begin{equation*}
Df_{m,n,k}^{j}=2\tilde{N}\left\vert X_{j}\right\vert ^{2(\tilde{N}%
-1)}\langle X_{j},DX_{j}\rangle
\end{equation*}%
so that%
\begin{equation*}
\left\vert Df_{m,n,k}^{j}\right\vert _{\boldsymbol{H}}\leq 2\tilde{N}%
\left\vert X_{j}\right\vert ^{2\tilde{N}-1}\left\vert DX_{j}\right\vert _{%
\boldsymbol{H}}\text{ .}
\end{equation*}%
Thus, choose $\alpha >1$ such that $\beta =\alpha q(2\tilde{N}-1)>1$. Then,
by using H\"{o}lder's inequality%
\begin{equation}
\left\Vert Df_{m,n,k}^{j}\right\Vert _{q}\leq 2\tilde{N}\left\Vert
X_{j}\right\Vert _{\beta }^{2\tilde{N}-1}\left\Vert DX_{j}\right\Vert
_{q\alpha ^{\prime }}  \label{m23-3}
\end{equation}%
where $\frac{1}{\alpha }+\frac{1}{\alpha ^{\prime }}=1$. Using (\ref{m22-4})
and (\ref{m23-1}) to obtain%
\begin{equation*}
\left\Vert Df_{m,n,k}^{1}\right\Vert _{q}\leq C\left( \sqrt{q}\sqrt{\frac{%
2^{m}}{2^{n}}}\sqrt{\frac{1}{2^{n}}}\right) ^{2\tilde{N}}
\end{equation*}%
and%
\begin{equation*}
\left\Vert Df_{m,n,k}^{2}\right\Vert _{q}\leq \left\{ 
\begin{array}{cc}
C\left( \sqrt{q}^{2}\sqrt{\frac{1}{2^{m+n}}}\right) ^{2\tilde{N}-1}\left( 
\sqrt{q}\sqrt{\frac{1}{2^{m+n}}}\right) & \text{\ if }n\leq m\text{,} \\ 
C\left( \sqrt{q}^{2}\frac{2^{m}}{2^{2n}}\text{ }\right) ^{2\tilde{N}%
-1}\left( \sqrt{q}\sqrt{\frac{2^{3m}}{2^{3n}}}\sqrt{\frac{1}{2^{n+m}}}\right)
& \text{if }n>m%
\end{array}%
\right.
\end{equation*}%
thus (\ref{m23-6}) and (\ref{m23-5}) follow immediately.
\end{proof}

\subsection{Estimating capacities}

\begin{lemma}
\label{le3}There is a constant $C$ depending only on $p,\gamma $ and $d$
such that%
\begin{equation}
\left\Vert \rho _{j}(\boldsymbol{w}^{(m)})^{\frac{p}{j}}\right\Vert _{q}\leq
Cq^{\frac{p}{2}}\text{ \ \ \ \ \ }\forall m\in \mathbb{N}\text{, }q\geq 1
\label{m27-5}
\end{equation}%
where $j=1,2$.
\end{lemma}

\begin{proof}
(\ref{m27-5}) follows from the following inequality%
\begin{equation*}
\left\Vert \rho _{j}(\boldsymbol{w}^{(m)})^{\frac{p}{j}}\right\Vert _{q}\leq
\sum_{n=1}^{\infty }n^{\gamma }\sum_{k=1}^{2^{n}}\left\Vert
w_{t_{n}^{k-1},t_{n}^{k}}^{(m),j}\right\Vert _{pq}^{p}
\end{equation*}%
and Lemma \ref{lem1a} applying to the $L^{pq}$-norm.
\end{proof}

\begin{lemma}
\label{le4}There is a constant $C$ depending only on $p$, $\gamma $ and $d$
such that%
\begin{equation}
\left\Vert \rho _{j}(\boldsymbol{w}^{(m+1)},\boldsymbol{w}^{(m)})^{\frac{p}{j%
}}\right\Vert _{q}\leq Cq^{\frac{p}{2}}\left( \frac{1}{2^{m}}\right) ^{\frac{%
p-2}{4}}  \label{m27-7}
\end{equation}%
for all $m\in \mathbb{N}$, and $j=1,2$ and $q>1$.
\end{lemma}

\begin{proof}
thus, together with (\ref{4-17-6}), there is a constant depending only on $d$
such that Therefore, by applying these estimates to $L^{pq}$ norms we deduce
that%
\begin{eqnarray*}
\left\Vert \rho _{1}(\boldsymbol{w}^{(m+1)},\boldsymbol{w}%
^{(m)})^{p}\right\Vert _{q} &\leq &C\sum_{n>m}^{\infty }n^{\gamma
}\sum_{k=1}^{2^{n}}\left\Vert
w_{t_{n}^{k-1},t_{n}^{k}}^{(m+1),1}-w_{t_{n}^{k-1},t_{n}^{k}}^{(m),1}\right%
\Vert _{qp}^{p} \\
&\leq &C\sum_{n>m}^{\infty }n^{\gamma }2^{n}\left( \sqrt{q}\frac{\sqrt{2^{m}}%
}{2^{n}}\right) ^{p} \\
&\leq &Cq^{\frac{p}{2}}\sum_{n>m}^{\infty }n^{\gamma }\left( \frac{1}{2^{n}}%
\right) ^{\frac{p-2}{2}} \\
&\leq &Cq^{\frac{p}{2}}\left( \frac{1}{2^{m}}\right) ^{\frac{p-2}{4}}\text{.}
\end{eqnarray*}%
Similarly%
\begin{eqnarray*}
\left\Vert \rho _{2}(\boldsymbol{w}^{(m+1)},\boldsymbol{w}^{(m)})^{\frac{p}{2%
}}\right\Vert _{q} &\leq &\sum_{n=1}^{\infty }n^{\gamma
}\sum_{k=1}^{2^{n}}\left\Vert
w_{t_{n}^{k-1},t_{n}^{k}}^{(m+1),2}-w_{t_{n}^{k-1},t_{n}^{k}}^{(m),j}\right%
\Vert _{q\frac{p}{2}}^{\frac{p}{2}} \\
&\leq &Cq^{\frac{p}{2}}\left( \sum_{n\leq m}^{\infty }n^{\gamma }2^{n}\sqrt{%
\frac{1}{2^{m+n}}}^{\frac{p}{2}}+\sum_{n>m}^{\infty }n^{\gamma }2^{n}\frac{%
2^{m}}{2^{n}}\left( \frac{1}{2^{n}}\right) ^{\frac{p}{2}}\right) \\
&\leq &Cq^{\frac{p}{2}}\sum_{n\leq m}^{\infty }n^{\gamma }\sqrt{\frac{2^{n}}{%
2^{m}}}\left( \frac{1}{2^{m}}\right) ^{\frac{p-2}{4}}\left( \frac{1}{2^{n}}%
\right) ^{\frac{p-2}{4}}+\sum_{n>m}^{\infty }n^{\gamma }\left( \frac{1}{2^{n}%
}\right) ^{\frac{p}{2}-1} \\
&\leq &Cq^{\frac{p}{2}}\left( \frac{1}{2^{m}}\right) ^{\frac{p-2}{4}}\text{.}
\end{eqnarray*}
\end{proof}

\begin{lemma}
\label{le5}There is a constant $C$ depending only on $p$, $\gamma $ and $d$
such that%
\begin{equation}
\left\Vert \rho _{1}(\boldsymbol{w}^{(m)})^{p}\right\Vert _{q,1}\leq Cq^{%
\frac{p}{2}}  \label{m27-8}
\end{equation}%
for any $q\geq 1$ and $m\in N$.
\end{lemma}

\begin{proof}
Let $Y_{1}=w_{t_{n}^{k-1},t_{n}^{k}}^{(m),1}$. Then, by (\ref{m-02}), $%
DY_{1}=\boldsymbol{1}_{J_{n}^{k}}$ (if $n<m$) or $\frac{2^{m}}{2^{n}}%
\boldsymbol{1}_{J_{m}^{k(n,m)}}$ for $n\geq m$, so that%
\begin{eqnarray*}
\left\Vert D\left\vert Y_{1}\right\vert ^{p}\right\Vert _{q} &=&p\left\Vert
Y_{1}\right\Vert _{q(p-1)}^{p-1}\left\vert DY_{1}\right\vert _{\boldsymbol{H}%
} \\
&=&\left\{ 
\begin{array}{cc}
\left( \frac{1}{\sqrt{2^{n}}}||\xi ||_{(p-1)q}\right) ^{p-1}\sqrt{\frac{1}{%
2^{n}}} & \text{for }n<m \\ 
\left( \frac{\sqrt{2^{m}}}{2^{n}}||\xi ||_{(p-1)q}\right) ^{p-1}\frac{2^{m}}{%
2^{n}}\sqrt{\frac{1}{2^{m}}} & \text{for }n\geq m\text{ }%
\end{array}%
\right.
\end{eqnarray*}%
where $\xi \sim N(0,1_{\mathbb{R}^{d}})$. Hence%
\begin{eqnarray*}
\left\Vert \left\vert D\rho _{1}(\boldsymbol{w}^{(m)})^{p}\right\vert _{%
\boldsymbol{H}}\right\Vert _{q} &\leq &\sum_{n=1}^{\infty }n^{\gamma
}\sum_{k=1}^{2^{n}}\left\Vert D\left\vert
w_{t_{n}^{k-1},t_{n}^{k}}^{(m),1}\right\vert ^{p}\right\Vert _{q} \\
&\leq &Cq^{\frac{p-1}{2}}\sum_{n=1}^{\infty }n^{\gamma }\left( \frac{1}{2^{n}%
}\right) ^{\frac{p}{2}-1}\text{.}
\end{eqnarray*}
\end{proof}

\begin{lemma}
\label{le8}There is a constant $C$ depending only on $p$, $\gamma $ and $d$
such that%
\begin{equation}
\left\Vert \rho _{j}(\boldsymbol{w}^{(m+1)},\boldsymbol{w}^{(m)})^{\frac{p}{j%
}}\right\Vert _{q,1}\leq Cq^{\frac{p}{2}}\left( \frac{1}{2^{m}}\right) ^{%
\frac{p-2}{4}}  \label{m23-7a}
\end{equation}%
for any $m$, $q\geq 1$ and $j=1,2$.
\end{lemma}

\begin{proof}
Let $u_{m,n,k}^{j}(w)=\left\vert
w_{t_{n}^{k-1},t_{n}^{k}}^{(m+1),j}-w_{t_{n}^{k-1},t_{n}^{k}}^{(m),j}\right%
\vert ^{\frac{p}{j}}$. According to (\ref{m22-4}), (\ref{m22-05}), (\ref%
{m23-6}) and (\ref{m23-5}) with $2\tilde{N}=\frac{p}{j}$, we may deduce that 
\begin{equation*}
\left\Vert u_{m,n,k}^{1}\right\Vert _{q,1}\leq Cq^{\frac{p}{2}}\left( \frac{%
2^{m}}{2^{2n}}\right) ^{\frac{p}{2}}\text{\ for }n>m
\end{equation*}%
and%
\begin{equation*}
\left\Vert u_{m,n,k}^{2}\right\Vert _{q,1}\leq \left\{ 
\begin{array}{cc}
Cq^{\frac{p}{2}}\left( \frac{1}{2^{m+n}}\right) ^{\frac{p}{4}} & \text{\ if }%
n\leq m\text{,} \\ 
Cq^{\frac{p}{2}}\left( \frac{2^{m}}{2^{2n}}\text{ }\right) ^{\frac{p}{2}} & 
\text{if }n>m\text{ .}%
\end{array}%
\right.
\end{equation*}%
On the other hand, by triangle inequality 
\begin{equation*}
\left\Vert \rho _{j}(\boldsymbol{w}^{(m+1)},\boldsymbol{w}^{(m)})^{\frac{p}{j%
}}\right\Vert _{q,1}\leq \sum_{n=1}^{\infty }n^{\gamma
}\sum_{k=1}^{2^{n}}||u_{m,n,k}^{j}||_{q,1}\text{.}
\end{equation*}%
It follows thus that%
\begin{eqnarray*}
\left\Vert \rho _{1}(\boldsymbol{w}^{(m+1)},\boldsymbol{w}%
^{(m)})^{p}\right\Vert _{q,1} &\leq &Cq^{\frac{p}{2}}\sum_{n>m}^{\infty
}n^{\gamma }2^{n}\left( \frac{1}{2^{n}}\right) ^{\frac{p-2}{2}} \\
&\leq &Cq^{\frac{p}{2}}\left( \frac{1}{2^{m}}\right) ^{\frac{p-2}{4}%
}\sum_{n>m}^{\infty }n^{\gamma }\left( \frac{1}{2^{n}}\right) ^{\frac{p-2}{4}%
} \\
&\leq &Cq^{\frac{p}{2}}\left( \frac{1}{2^{m}}\right) ^{\frac{p-2}{4}}\text{.}
\end{eqnarray*}%
and%
\begin{eqnarray*}
\left\Vert D\rho _{2}(\boldsymbol{w}^{(m+1)},\boldsymbol{w}^{(m)})^{\frac{p}{%
2}}\right\Vert _{q} &\leq &Cq^{\frac{p}{2}}\left( \sum_{n\leq m}^{\infty
}n^{\gamma }2^{n}\left( \frac{1}{2^{m+n}}\right) ^{\frac{p}{4}%
}+\sum_{n>m}^{\infty }n^{\gamma }2^{n}\left( \frac{2^{m}}{2^{2n}}\text{ }%
\right) ^{\frac{p}{2}}\right) \\
&\leq &Cq^{\frac{p}{2}}\left( \frac{1}{2^{m}}\right) ^{\frac{p-2}{4}%
}\sum_{n=1}^{\infty }n^{\gamma }\left( \frac{1}{2^{n}}\right) ^{\frac{p-2}{4}%
}\text{.}
\end{eqnarray*}%
where constants $C$ may be different from line to line but only depend on $%
p,\gamma $ and $d$.
\end{proof}

Finally we need an $L^{q}$-estimate for the\ malliavin derivative of $\rho
_{1}(\boldsymbol{w}^{(m)})^{p}\rho _{1}(\boldsymbol{w}^{(m+1)},\boldsymbol{w}%
^{(m)})^{p}$.

\begin{lemma}
\label{le9}There is a constant $C$ depending only on $p$, $\gamma $ and $d$
such that%
\begin{equation}
\left\Vert \rho _{1}(\boldsymbol{w}^{(m)})^{p}\rho _{1}(\boldsymbol{w}%
^{(m+1)},\boldsymbol{w}^{(m)})^{p}\right\Vert _{q,1}\leq Cq^{\frac{p}{2}%
}\left( \frac{1}{2^{m}}\right) ^{\frac{p-2}{8}}  \label{m27-9}
\end{equation}%
for any $m\in \mathbb{N}$.
\end{lemma}

\begin{proof}
Firstly, by Cauchy-Schwartz's inequality%
\begin{eqnarray*}
&&\left\Vert \rho _{1}(\boldsymbol{w}^{(m)})^{p}\rho _{1}(\boldsymbol{w}%
^{(m+1)},\boldsymbol{w}^{(m)})^{p}\right\Vert _{q} \\
&\leq &\sqrt{\left\Vert \rho _{1}(\boldsymbol{w}^{(m)})^{p}\right\Vert _{2q}}%
\sqrt{\left\Vert \rho _{1}(\boldsymbol{w}^{(m+1)},\boldsymbol{w}%
^{(m)})^{p}\right\Vert _{2q}} \\
&\leq &Cq^{\frac{p}{2}}\left( \frac{1}{2^{m}}\right) ^{\frac{p-2}{8}}\text{.}
\end{eqnarray*}%
Similarly, by chain rule, and Cauchy-Schwartz inequality%
\begin{eqnarray*}
&&\left\Vert D\left( \rho _{1}(\boldsymbol{w}^{(m)})^{p}\rho _{1}(%
\boldsymbol{w}^{(m+1)},\boldsymbol{w}^{(m)})^{p}\right) \right\Vert _{q} \\
&\leq &\left\Vert \rho _{1}(\boldsymbol{w}^{(m)})^{p}D\rho _{1}(\boldsymbol{w%
}^{(m+1)},\boldsymbol{w}^{(m)})^{p}\right\Vert _{q} \\
&&+\left\Vert \rho _{1}(\boldsymbol{w}^{(m+1)},\boldsymbol{w}%
^{(m)})^{p}D\rho _{1}(\boldsymbol{w}^{(m)})^{p}\right\Vert _{q} \\
&\leq &\sqrt{\left\Vert \rho _{1}(\boldsymbol{w}^{(m)})^{p}\right\Vert _{2q}}%
\sqrt{\left\Vert D\rho _{1}(\boldsymbol{w}^{(m+1)},\boldsymbol{w}%
^{(m)})^{p}\right\Vert _{2q}} \\
&&+\sqrt{\left\Vert D\rho _{1}(\boldsymbol{w}^{(m)})^{p}\right\Vert _{2q}}%
\sqrt{\left\Vert \rho _{1}(\boldsymbol{w}^{(m+1)},\boldsymbol{w}%
^{(m)})^{p}\right\Vert _{2q}}\text{,}
\end{eqnarray*}%
together with (\ref{m27-5}), (\ref{m27-8}) and (\ref{m23-7a}) we obtain (\ref%
{m27-9}).
\end{proof}

Recall that $p\in (2,3)$ and $\gamma >\frac{p}{2}-1$, and $C_{p,\gamma }$
the constant appearing in (\ref{m29-9}).

\begin{theorem}
\label{th8}Suppose $\beta \in (0,\frac{p-2}{8p})$, then%
\begin{equation}
\sum_{m}\text{Cap}_{q,1}\left\{ d_{p}(\boldsymbol{w}^{(m+1)},\boldsymbol{w}%
^{(m)})|>C_{p,\gamma }\left( \frac{1}{2^{m}}\right) ^{\beta }\right\}
<\infty \text{ .}  \label{tr-01}
\end{equation}
\end{theorem}

\begin{proof}
By using our basic estimate (\ref{m23-7a})%
\begin{eqnarray*}
\text{\textit{Cap}}_{q,1}\left\{ \rho _{j}(\boldsymbol{w}^{(m+1)},%
\boldsymbol{w}^{(m)})>\lambda ^{\frac{j}{p}}\right\} &=&\text{\textit{Cap}}%
_{q,1}\left\{ \rho _{j}(\boldsymbol{w}^{(m+1)},\boldsymbol{w}^{(m)})^{\frac{p%
}{j}}>\lambda \right\} \\
&\leq &\frac{1}{\lambda }\left\Vert \rho _{j}(\boldsymbol{w}^{(m+1)},%
\boldsymbol{w}^{(m)})^{\frac{p}{j}}\right\Vert _{q,1} \\
&\leq &\frac{C}{\lambda }\left( \frac{1}{2^{m}}\right) ^{\frac{p-2}{4}}
\end{eqnarray*}%
where $C$ is a constant depending only on $d$ and $p$. Choose $\lambda $
such that%
\begin{equation*}
\lambda ^{j/p}=\left( \frac{1}{2^{m}}\right) ^{\beta }
\end{equation*}%
to obtain%
\begin{equation*}
\text{\textit{Cap}}_{q,1}\left\{ \rho _{j}(\boldsymbol{w}^{(m+1)},%
\boldsymbol{w}^{(m)})>\left( \frac{1}{2^{m}}\right) ^{\beta }\right\} \leq
C\left( \frac{1}{2^{m}}\right) ^{\frac{p-2}{4}-\frac{p\beta }{j}}\text{.}
\end{equation*}%
Since $\frac{p-2}{4}-\frac{p\beta }{j}>0$ so that 
\begin{equation*}
\sum_{m}\text{\textit{Cap}}_{q,1}\left\{ |\rho _{j}(\boldsymbol{w}^{(m+1)},%
\boldsymbol{w}^{(m)})|>\left( \frac{1}{2^{m}}\right) ^{\beta }\right\}
<\infty \text{.}
\end{equation*}%
Similarly 
\begin{eqnarray*}
&&\text{\textit{Cap}}_{q,1}\left( |\rho _{1}(\boldsymbol{w}^{(m)})\rho _{1}(%
\boldsymbol{w}^{(m+1)},\boldsymbol{w}^{(m)})|>\lambda ^{1/p}\right) \\
&\leq &\frac{1}{\lambda }\left\Vert \rho _{1}(\boldsymbol{w}^{(m)})^{p}\rho
_{1}(\boldsymbol{w}^{(m+1)},\boldsymbol{w}^{(m)})^{p}\right\Vert _{q,1} \\
&\leq &\frac{C}{\lambda }\left( \frac{1}{2^{m}}\right) ^{\frac{p-2}{8}}
\end{eqnarray*}%
so that%
\begin{equation*}
\text{\textit{Cap}}_{q,1}\left\{ \rho _{j}(\boldsymbol{w}^{(m+1)},%
\boldsymbol{w}^{(m)})>\left( \frac{1}{2^{m}}\right) ^{\beta }\right\} \leq
C\left( \frac{1}{2^{m}}\right) ^{\frac{p-2}{8}-p\beta }\text{.}
\end{equation*}%
Now (\ref{tr-01}) follows from (\ref{m29-9}).
\end{proof}

\begin{corollary}
Suppose $p\in (2,3)$, then%
\begin{equation*}
\text{\textit{Cap}}_{q,1}\left\{ \sum_{m=1}^{\infty }d_{p}(\boldsymbol{w}%
^{(m+1)},\boldsymbol{w}^{(m)})=\infty \right\} =0\text{, \ \ \ }\forall
q\geq 1\text{.}
\end{equation*}
\end{corollary}

We have thus proved (\ref{m28-4}) for $N=1$.

\section{The proof of the quasi-sure convergence}

Guided by the estimates we have obtained in the previous section, we wish to
show that for every pair $q\geq 1$ and $N\in \mathbb{N}$ 
\begin{equation*}
\text{\textit{Cap}}_{q,N}\left\{ \rho _{j}(\boldsymbol{w}^{(m+1)},%
\boldsymbol{w}^{(m)})^{\frac{p}{j}}>C\left( \frac{1}{2^{m}}\right) ^{\beta
}\right\} \leq C^{\prime }\left( \frac{1}{2^{m}}\right) ^{\varepsilon }
\end{equation*}%
for some choices of $\beta >0$ and $\varepsilon >0$, where $C$ and $%
C^{\prime }$ are two constants independent of $m$.

Therefore we are interested in the capacity%
\begin{equation*}
I_{j}(m)=\text{\textit{Cap}}_{q,N}\left\{ \rho _{j}(\boldsymbol{w}^{(m+1)},%
\boldsymbol{w}^{(m)})^{\frac{p}{j}}>\lambda \right\}
\end{equation*}%
($j=1,2$).

Since 
\begin{equation*}
\rho _{j}(\boldsymbol{w}^{(m+1)},\boldsymbol{w}^{(m)})^{\frac{p}{j}%
}=\sum_{n=1}^{\infty }n^{\gamma }\sum_{k=1}^{2^{n}}\left\vert
w_{t_{n}^{k-1},t_{n}^{k}}^{(m+1),j}-w_{t_{n}^{k-1},t_{n}^{k}}^{(m),j}\right%
\vert ^{\frac{p}{j}}
\end{equation*}%
so that, for every $\theta >0$%
\begin{equation}
\left\{ \rho _{j}(\boldsymbol{w}^{(m+1)},\boldsymbol{w}^{(m)})^{\frac{p}{j}%
}>\lambda \right\} \subset \dbigcup\limits_{n=1}^{\infty }\left\{
\sum_{k=1}^{2^{n}}\left\vert
w_{t_{n}^{k-1},t_{n}^{k}}^{(m+1),j}-w_{t_{n}^{k-1},t_{n}^{k}}^{(m),j}\right%
\vert ^{\frac{p}{j}}>C_{\theta }\lambda \left( \frac{1}{2^{n}}\right)
^{\theta }\right\}  \label{4-22-a1}
\end{equation}%
where 
\begin{equation*}
C_{\theta }=\frac{1}{\sum_{n=1}^{\infty }n^{\gamma }\left( \frac{1}{2^{n}}%
\right) ^{\theta }}\text{. }
\end{equation*}%
Therefore%
\begin{eqnarray}
I_{j}(m) &\leq &\sum_{n=1}^{\infty }\text{\textit{Cap}}_{q,N}\left\{
\sum_{k=1}^{2^{n}}\left\vert
w_{t_{n}^{k-1},t_{n}^{k}}^{(m+1),j}-w_{t_{n}^{k-1},t_{n}^{k}}^{(m),j}\right%
\vert ^{\frac{p}{j}}>\lambda C_{\theta }\left( \frac{1}{2^{n}}\right)
^{\theta }\right\}  \notag \\
&\leq &\sum_{n=1}^{\infty }\sum_{k=1}^{2^{n}}\text{\textit{Cap}}%
_{q,N}\left\{ \left\vert
w_{t_{n}^{k-1},t_{n}^{k}}^{(m+1),j}-w_{t_{n}^{k-1},t_{n}^{k}}^{(m),j}\right%
\vert ^{\frac{p}{j}}>\lambda C_{\theta }\left( \frac{1}{2^{n}}\right)
^{\theta +1}\right\} \text{.}  \label{4-22-a2}
\end{eqnarray}

On the other hand, for any $\tilde{N}>0$ we have 
\begin{eqnarray*}
&&\text{\textit{Cap}}_{q,N}\left\{ \left\vert
w_{t_{n}^{k-1},t_{n}^{k}}^{(m+1),j}-w_{t_{n}^{k-1},t_{n}^{k}}^{(m),j}\right%
\vert ^{\frac{p}{j}}>\lambda C_{\theta }\left( \frac{1}{2^{n}}\right)
^{\theta +1}\right\} \\
&=&\text{\textit{Cap}}_{q,N}\left\{ \left\vert
w_{t_{n}^{k-1},t_{n}^{k}}^{(m+1),j}-w_{t_{n}^{k-1},t_{n}^{k}}^{(m),j}\right%
\vert >\lambda ^{\frac{j}{p}}C_{\theta }^{\frac{j}{p}}\left( \frac{1}{2^{n}}%
\right) ^{\frac{j}{p}(\theta +1)}\right\} \\
&=&\text{ \textit{Cap}}_{q,N}\left\{ f_{m,n,k}^{j}>\left[ \lambda ^{\frac{j}{%
p}}C_{\theta }^{\frac{j}{p}}\left( \frac{1}{2^{n}}\right) ^{\frac{j}{p}%
(\theta +1)}\right] ^{2\tilde{N}}\right\}
\end{eqnarray*}%
where 
\begin{equation*}
f_{m,n,k}^{j}(w)=\left\vert
w_{t_{n}^{k-1},t_{n}^{k}}^{(m+1),j}-w_{t_{n}^{k-1},t_{n}^{k}}^{(m),j}\right%
\vert ^{2\tilde{N}}\text{ \ \ for }w\in \boldsymbol{W}\text{ .}
\end{equation*}%
If $\tilde{N}$ is a natural number, then $f_{m,n,k}^{j}$ are polynomials of
the Brownian motion paths, so are smooth functionals on the Wiener space $%
\boldsymbol{W}$ in Malliavin's sense. This latter fact allows us to apply
the capacity maximal inequality to bound the preceding capacity. Namely, for
each pair $q\geq 1$ and $N\in \mathbb{N}$, we have%
\begin{equation}
\text{\textit{Cap}}_{q,N}\left\{ \left\vert
w_{t_{n}^{k-1},t_{n}^{k}}^{(m+1),j}-w_{t_{n}^{k-1},t_{n}^{k}}^{(m),j}\right%
\vert ^{\frac{p}{j}}>\lambda C_{\theta }\left( \frac{1}{2^{n}}\right)
^{\theta +1}\right\} \leq C\left[ \lambda ^{\frac{j}{p}}C_{\theta }^{\frac{j%
}{p}}\left( \frac{1}{2^{n}}\right) ^{\frac{j}{p}(\theta +1)}\right] ^{-2%
\tilde{N}}||f_{m,n,k}^{j}||_{q,N}  \label{4-22-3a}
\end{equation}%
where $C$ depends only on $d,q$ and $N$. It thus follows that%
\begin{equation}
I_{j}(m)\leq C\sum_{n=1}^{\infty }\sum_{k=1}^{2^{n}}\left[ \lambda ^{\frac{j%
}{p}}C_{\theta }^{\frac{j}{p}}\left( \frac{1}{2^{n}}\right) ^{\frac{j}{p}%
(\theta +1)}\right] ^{-2\tilde{N}}||f_{m,n,k}^{j}||_{q,N}\text{.}
\label{4-22-4a}
\end{equation}

Therefore, we need to estimate the Sobolev norm $||f_{m,n,k}^{j}||_{q,N}$ in
order to prove our main theorem \ref{th5}, and we will see that there is a
good reason (see the constraint (\ref{4-24-1}) below) that we need to raise
the power of $%
|w_{t_{n}^{k-1},t_{n}^{k}}^{(m+1),j}-w_{t_{n}^{k-1},t_{n}^{k}}^{(m),j}|^{%
\frac{p}{j}}$ $\ $to $2\tilde{N}$ for large enough $\tilde{N}$.

To this end, we first need to evaluate higher order Malliavin derivatives of 
$f_{m,n,k}^{j}$. For simplicity, let $%
X_{j}(w)=w_{t_{n}^{k-1},t_{n}^{k}}^{(m+1),j}-w_{t_{n}^{k-1},t_{n}^{k}}^{(m),j}
$ and $Y_{j}(w)=w_{t_{n}^{k-1},t_{n}^{k}}^{(m),j}$ ($j=1,2$). Suppose $%
\tilde{N}\in \mathbb{N}$ is chosen, and consider $f=|X|^{2\tilde{N}}$ where $%
X=X_{j}$ or $Y_{j}$ ($j=1,2$). In all these cases, $f$ is a polynomial of
Brownian motion path, and thus is smooth in the Malliavin sense. In
particular, $f\in \mathbb{D}_{N}^{q}$ for any $q\geq 1$ and $N\in \mathbb{N}$%
. We want to find an upper bound for the Sobolev norm $||f||_{q,N}$.

If $M\leq \tilde{N}$, we have%
\begin{equation}
D^{M}f=\sum_{\mu =1}^{M}\dsum\limits_{\substack{ \alpha _{1}+\cdots +\alpha
_{\mu }=M \\ 4\geq \alpha _{i}\geq 1}}C_{\alpha _{1}\cdots \alpha _{\mu
}}|X|^{2(\tilde{N}-\mu )}D^{\alpha _{1}}|X|^{2}\otimes \cdots \otimes
D^{\alpha _{\mu }}|X|^{2}  \label{4-17-9}
\end{equation}%
where $C_{\alpha _{1}\cdots \alpha _{\mu }}$ are constants depending only on 
$\alpha $'s, $j=1$ or $2$, $M$ and $\tilde{N}$. Therefore, by using H\"{o}%
lder's inequality, we have%
\begin{eqnarray}
||f||_{q,N} &\leq &C\sum_{M=0}^{N}\sum_{\mu =1}^{M}\dsum\limits_{\substack{ %
\alpha _{1}+\cdots +\alpha _{\mu }=M \\ 4\geq \alpha _{i}\geq 0}}|||X|^{2(%
\tilde{N}-\mu )}D^{\alpha _{1}}|X|^{2}\otimes \cdots \otimes D^{\alpha _{\mu
}}|X|^{2}||_{q}  \notag \\
&\leq &C\sum_{M=0}^{N}\sum_{\mu =1}^{M}\dsum\limits_{\substack{ a_{1}+\cdots
+a_{\mu }=M \\ 4\geq a_{i}\geq 0}}||X||_{4q(\tilde{N}-\mu )}^{2(\tilde{N}%
-\mu )}\prod_{i=1}^{\mu }||D^{\alpha _{i}}|X|^{2}||_{2\mu q}  \label{4-18-2}
\end{eqnarray}%
for some constant $C$ depending only on $N,\tilde{N}$, where the restriction
that $4\geq a_{i}\geq 0$ comes from the fact that $|X|^{2}$ is a polynomial
of Brownian motion of order at most four as $X=X_{j}$ or $Y_{j}$, so that $%
D^{a}|X|^{2}=0$ for $a\geq 5$. The inequality (\ref{4-18-2}), though
completely elementary, allows us to develop the necessary estimates for the
Sobolev norms we are interested.

\begin{lemma}
\label{lem-19-1}Suppose $X$ is a smooth Malliavin functional, and $D^{a}X=0$
for $a\geq 3$, then 
\begin{equation}
|D|X|^{2}|_{\boldsymbol{H}}\leq 2|X||DX|_{\boldsymbol{H}}\text{, }%
|D^{2}|X|^{2}|_{\boldsymbol{H}^{\otimes 2}}\leq 2|DX|_{\boldsymbol{H}%
}^{2}+2|X||D^{2}X|_{\boldsymbol{H}^{\otimes 2}}  \label{4-19-1}
\end{equation}%
and%
\begin{equation}
|D^{3}|X|^{2}|_{\boldsymbol{H}^{\otimes 3}}\leq 6|D^{2}X|_{\boldsymbol{H}%
^{\otimes 2}}|DX|_{\boldsymbol{H}}\text{, \ \ }|D^{4}|X|^{2}|_{\boldsymbol{H}%
^{\otimes 4}}\leq 6|D^{2}X|_{\boldsymbol{H}^{\otimes 2}}^{2}\text{.}
\label{4-19-2}
\end{equation}
\end{lemma}

\begin{proof}
These estimates follows from the chain rule directly.
\end{proof}

\begin{lemma}
Let $m,n\in N$, and $k=1,\cdots ,2^{n}$. Let\ $%
Y_{j}=w_{t_{n}^{k-1},t_{n}^{k}}^{(m),j}$ and $%
X_{j}=w_{t_{n}^{k-1},t_{n}^{k}}^{(m+1),j}-w_{t_{n}^{k-1},t_{n}^{k}}^{(m),j}$
\ ($j=1,2$). Then, there is a constant $C$ depending only on $d$, such that
for any $q\geq 1$%
\begin{equation}
\left\Vert D^{a}|Y_{1}|^{2}\right\Vert _{q}\leq \left\{ 
\begin{array}{cc}
C\sqrt{q}\frac{2^{m}}{2^{2n}} & \text{ \ if \ }n>m\text{,} \\ 
C\sqrt{q}\frac{1}{2^{n}} & \text{if }n\leq m\text{,}%
\end{array}%
\right. \text{ }  \label{4-19-3}
\end{equation}%
\begin{equation}
\left\Vert D^{a}|X_{1}|^{2}\right\Vert _{q}\leq \left\{ 
\begin{array}{cc}
C\sqrt{q}\frac{2^{m}}{2^{2n}} & \text{ \ if \ }n>m\text{,} \\ 
0 & \text{if }n\leq m%
\end{array}%
\right.  \label{4-19-4}
\end{equation}%
and%
\begin{equation}
\left\Vert D^{b}|X_{2}|^{2}\right\Vert _{q}\leq \left\{ 
\begin{array}{cc}
C\sqrt{q}^{4-b}\frac{2^{2m}}{2^{4n}}\text{\ \ } & \text{\ \ for }n\geq m%
\text{, } \\ 
C\sqrt{q}^{4-b}\frac{1}{2^{m+n}} & \text{for \ }n<m%
\end{array}%
\right. \text{ \ \ \ }  \label{4-20-7}
\end{equation}%
where $a=1,2$ and $b=1,2,3,4$.
\end{lemma}

\begin{proof}
If $n\geq m$, then $Y_{1}=\frac{2^{m}}{2^{n}}\xi _{m}^{k(n,m)}$, $DY_{1}=%
\frac{2^{m}}{2^{n}}\boldsymbol{1}_{J_{m}^{k(n,m)}}$ and $D^{2}Y_{1}=0$, so
that 
\begin{equation}
|D|Y_{1}|^{2}|_{\boldsymbol{H}}\leq 2|Y_{1}||DY_{1}|_{\boldsymbol{H}}=2\frac{%
2^{m}}{2^{n}}\frac{2^{m}}{2^{n}}\sqrt{\frac{1}{2^{m}}}|\xi _{m}^{k(n,m)}|
\label{4-17-12}
\end{equation}%
which yields%
\begin{equation*}
\left\Vert D|Y_{1}|^{2}\right\Vert _{q}\leq C\sqrt{q}\frac{2^{m}}{2^{2n}}
\end{equation*}%
where $C$ depends only on $d$. Similarly%
\begin{equation*}
|D^{2}|Y_{1}|^{2}|_{\boldsymbol{H}^{\otimes 2}}\leq 2|DY_{1}|_{\boldsymbol{H}%
}^{2}=2\frac{2^{2m}}{2^{2n}}\boldsymbol{1}_{J_{m}^{k(n,m)}}
\end{equation*}%
so that 
\begin{equation*}
\left\Vert D^{2}\left\vert Y_{1}\right\vert ^{2}\right\Vert _{q}\leq 2\frac{%
2^{m}}{2^{2n}}\text{.}
\end{equation*}%
If $n<m$, then $Y_{1}=\xi _{n}^{k}$ so that $DY_{1}=\boldsymbol{1}%
_{J_{n}^{k}}$, hence 
\begin{equation*}
\left\Vert D|Y_{1}|^{2}\right\Vert _{q}\leq C\sqrt{q}\frac{1}{2^{n}}\text{,
\ }\left\Vert D^{2}|Y_{1}|^{2}\right\Vert _{q}\leq C\frac{1}{2^{n}}
\end{equation*}%
where $C$ depends only on $d$. This proves (\ref{4-19-3}). (\ref{4-19-4})
follows (\ref{4-19-3}) and the fact that $X_{1}=0$ if $n\leq m$.

Together with Lemma \ref{lem-19-1} and the $L^{q}$-estimate (\ref{m22-05})
for $X_{2}$, we can conclude that there is a constant $C$ depending only on $%
d$ such that%
\begin{equation}
\left\Vert D^{\alpha }|X_{2}|^{2}\right\Vert _{q}\leq C\sqrt{q}^{4-\alpha }%
\frac{1}{2^{m+n}}\text{ \ \ \ }\forall n<m  \label{4-15-2}
\end{equation}%
for $\alpha =1,2,3,4$ and $q\geq 1$.

Now consider the case that $n>m$. In this case%
\begin{equation*}
Y_{2}=w_{t_{n}^{k-1},t_{n}^{k}}^{(m),2}=\frac{1}{2}\frac{2^{2m}}{2^{2n}}\xi
_{m}^{k(n,m)}\otimes \xi _{m}^{k(n,m)}\text{ }
\end{equation*}%
so that 
\begin{equation*}
DY_{2}=\frac{1}{2}\frac{2^{2m}}{2^{2n}}\{\boldsymbol{1}_{J_{m}^{k(n,m)}},\xi
_{m}^{k(n,m)}\}\text{, }D^{2}Y_{2}=\frac{1}{2}\frac{2^{2m}}{2^{2n}}\{%
\boldsymbol{1}_{J_{m}^{k(n,m)}},\boldsymbol{1}_{J_{m}^{k(n,m)}}\}\text{.}
\end{equation*}%
where 
\begin{equation*}
\{\boldsymbol{1}_{J_{m}^{k(n,m)}},\boldsymbol{1}_{J_{m}^{k(n,m)}}%
\}(t_{1},t_{2})=\{\boldsymbol{1}_{J_{m}^{k(n,m)}}(t_{1}),\boldsymbol{1}%
_{J_{m}^{k(n,m)}}(t_{2})\}\text{.}
\end{equation*}%
It follows that 
\begin{equation*}
||D^{b}|Y_{2}|^{2}||_{q}\leq Cq^{4-b}\frac{2^{2m}}{2^{4n}}\text{ \ \ \ \ for 
}n>m\text{, }b=1,2,3,4\text{.}
\end{equation*}%
and therefore (\ref{4-20-7}).
\end{proof}

In what follows we assume that $\tilde{N}\in \mathbb{N}$ and $N\leq \tilde{N}
$. Let%
\begin{equation}
f_{m,n,k}^{j}(w)=\left\vert w_{\frac{k-1}{2^{n}},\frac{k}{2^{n}}}^{\left(
m+1\right) ,j}-w_{\frac{k-1}{2^{n}},\frac{k}{2^{n}}}^{\left( m\right)
,j}\right\vert ^{2\tilde{N}}\text{,  }g_{m,n,k}^{j}(w)=\left\vert w_{\frac{%
k-1}{2^{n}},\frac{k}{2^{n}}}^{\left( m\right) ,j}\right\vert ^{2\tilde{N}}%
\text{.}  \label{4-22-1}
\end{equation}

\begin{lemma}
There is a constant $C$ depending only on $N,\tilde{N}$ and $d$ such that 
\begin{equation}
||f_{m,n,k}^{1}||_{q,N}\leq Cq^{\tilde{N}}\left( \frac{2^{m}}{2^{2n}}\right)
^{\tilde{N}}\text{ \ \ \ \ for }n>m  \label{4-17-11}
\end{equation}%
and%
\begin{equation}
||g_{m,n,k}^{1}||_{q,N}\leq \left\{ 
\begin{array}{cc}
Cq^{\tilde{N}}\left( \frac{2^{m}}{2^{2n}}\right) ^{\tilde{N}} & \text{ \ \
for }n>m \\ 
Cq^{\tilde{N}}\left( \frac{1}{2^{n}}\right) ^{\tilde{N}}\text{ \ \ \ } & 
\text{\ \ for }n\leq m%
\end{array}%
\right. \text{ \ \ \ \ }  \label{4-17-13}
\end{equation}%
for all $q\geq 1$.
\end{lemma}

\begin{proof}
If $X_{1}=w_{\frac{k-1}{2^{n}},\frac{k}{2^{n}}}^{\left( m+1\right) ,1}-w_{%
\frac{k-1}{2^{n}},\frac{k}{2^{n}}}^{\left( m\right) ,1}$ (for $n>m$,
otherwise $X_{1}=0$). By (\ref{4-18-2})%
\begin{eqnarray*}
||f_{m,n,k}^{1}||_{q,N} &\leq &C\sum_{M=0}^{N}\sum_{\mu =1}^{M}\dsum\limits 
_{\substack{ a_{1}+\cdots +a_{\mu }=M  \\ 2\geq a_{i}\geq 0}}||X_{1}||_{4q(%
\tilde{N}-\mu )}^{2(\tilde{N}-\mu )}\prod_{i=1}^{\mu }||D^{\alpha
_{i}}|X_{1}|^{2}||_{2\mu q} \\
&\leq &C\sum_{M=0}^{N}\sum_{\mu =1}^{M}\dsum\limits_{\substack{ a_{1}+\cdots
+a_{\mu }=M  \\ 4\geq a_{i}\geq 0}}\left( \sqrt{q}\frac{2^{m}}{2^{2n}}%
\right) ^{\mu }\left( q\frac{2^{m}}{2^{2n}}\right) ^{(\tilde{N}-\mu )} \\
&\leq &Cq^{\tilde{N}}\left( \frac{2^{m}}{2^{2n}}\right) ^{\tilde{N}}\text{ \
\ for }n>m\text{. }
\end{eqnarray*}%
The same estimate remains true in the case that $Y_{1}=w_{\frac{k-1}{2^{n}},%
\frac{k}{2^{n}}}^{\left( m\right) ,1}$ and $n>m$. On the other hand, if $%
n\leq m$, then 
\begin{eqnarray*}
||g_{m,n,k}^{1}||_{q,N} &\leq &C\sum_{M=0}^{N}\sum_{\mu =1}^{M}\dsum\limits 
_{\substack{ a_{1}+\cdots +a_{\mu }=M  \\ 2\geq a_{i}\geq 0}}||Y_{1}||_{4q(%
\tilde{N}-\mu )}^{2(\tilde{N}-\mu )}\prod_{i=1}^{\mu }||D^{\alpha
_{i}}|Y_{1}|^{2}||_{2\mu q} \\
&\leq &Cq^{\tilde{N}}\sum_{M=0}^{N}\sum_{\mu =1}^{M}\dsum\limits_{\substack{ %
a_{1}+\cdots +a_{\mu }=M  \\ 2\geq a_{i}\geq 0}}\left( \frac{1}{2^{n}}%
\right) ^{\mu }\left( \frac{1}{2^{n}}\right) ^{(\tilde{N}-\mu )} \\
&=&Cq^{\tilde{N}}\left( \frac{1}{2^{n}}\right) ^{\tilde{N}}\text{ \ \ \ \ \
for }n\leq m\text{. }
\end{eqnarray*}
\end{proof}

\begin{lemma}
There is a constant $C$ depending only on $N,\tilde{N},d$ such that%
\begin{equation}
||f_{m,n,k}^{2}||_{q,N}\leq \left\{ 
\begin{array}{cc}
C\sqrt{q}^{4\tilde{N}}\left( \frac{1}{2^{n}}\right) ^{\tilde{N}}\left( \frac{%
1}{2^{m}}\right) ^{\tilde{N}} & \text{ \ \ for }n\leq m\text{,} \\ 
Cq^{4\tilde{N}}\left( \frac{2^{m}}{2^{2n}}\right) ^{2\tilde{N}} & \text{ \
for }n>m%
\end{array}%
\right. \text{.}  \label{4-15-1}
\end{equation}%
for all $q\geq 1$.
\end{lemma}

\begin{proof}
Let $X_{2}=w_{\frac{k-1}{2^{n}},\frac{k}{2^{n}}}^{\left( m+1\right) ,j}-w_{%
\frac{k-1}{2^{n}},\frac{k}{2^{n}}}^{\left( m\right) ,j}$. By (\ref{4-18-2})
and the $L^{q}$-bounds of $X_{2}$ applying to $4q(\tilde{N}-\mu )$ and (\ref%
{4-20-7}) 
\begin{equation*}
\left\Vert X_{2}\right\Vert _{4q(\tilde{N}-\mu )}\leq \left\{ 
\begin{array}{cc}
Cq\frac{2^{m}}{2^{2n}} & \text{\ if }n>m\text{ \ \ } \\ 
\text{ \ }Cq\sqrt{\frac{1}{2^{m+n}}}\text{\ \ \ } & \text{ if }n\leq m\text{,%
}%
\end{array}%
\right.
\end{equation*}%
and%
\begin{equation*}
\left\Vert D^{b}|X_{2}|^{2}\right\Vert _{2\mu q}\leq \left\{ 
\begin{array}{cc}
Cq^{2}\frac{2^{2m}}{2^{4n}}\text{\ \ } & \text{\ \ for }n>m\text{ ,} \\ 
Cq^{2}\frac{1}{2^{m+n}} & \text{for \ }n\leq m\text{,}%
\end{array}%
\right. \text{ \ \ \ }
\end{equation*}
we obtain, for $n\leq m$,%
\begin{eqnarray*}
||f_{m,n,k}^{2}||_{q,N} &\leq &C\sum_{M=0}^{N}\sum_{\mu =1}^{M}\dsum\limits 
_{\substack{ a_{1}+\cdots +a_{\mu }=M  \\ 4\geq a_{i}\geq 0}}\left( q\sqrt{%
\frac{1}{2^{m+n}}}\right) ^{2(\tilde{N}-\mu )}\left( q^{2}\frac{1}{2^{m+n}}%
\right) ^{\mu } \\
&\leq &C\sqrt{q}^{4\tilde{N}}\left( \frac{1}{2^{m+n}}\right) ^{\tilde{N}}%
\text{.}
\end{eqnarray*}

Similarly, if $n>m$ then%
\begin{eqnarray*}
||f_{m,n,k}^{2}||_{q,N} &\leq &C\sum_{M=0}^{N}\sum_{\mu =1}^{M}\dsum\limits 
_{\substack{ a_{1}+\cdots +a_{\mu }=M  \\ 4\geq a_{i}\geq 0}}\left( q\frac{%
2^{m}}{2^{2n}}\right) ^{2(\tilde{N}-\mu )}\left( q^{2}\frac{2^{2m}}{2^{4n}}%
\right) ^{\mu } \\
&\leq &C\sqrt{q}^{4\tilde{N}}\left( \frac{2^{m}}{2^{2n}}\right) ^{2\tilde{N}}%
\text{.}
\end{eqnarray*}
\end{proof}

\begin{proposition}
\label{j2.1}Choose $\tilde{N}\in \mathbb{N}$ and $\theta ,\beta >0$ such that%
\begin{equation}
p-2-\frac{p}{\tilde{N}}>0\text{, }\beta +\theta <\frac{p-2-\frac{p}{\tilde{N}%
}}{2}  \label{4-24-1}
\end{equation}%
Then for any $N\leq \tilde{N}$ there is a constant $C$ depending only on $N,%
\tilde{N},\theta ,\beta ,d$ and $q\geq 1$ such that%
\begin{equation}
\text{\textit{Cap}}_{q,N}\left\{ \sum_{k=1}^{2^{n}}\left\vert
w_{t_{n}^{k-1},t_{n}^{k}}^{(m+1),j}-w_{t_{n}^{k-1},t_{n}^{k}}^{(m),j}\right%
\vert ^{\frac{p}{j}}>C_{\theta }\left( \frac{1}{2^{m}}\right) ^{\beta
}\left( \frac{1}{2^{n}}\right) ^{\theta }\right\} \leq C\left( \frac{1}{%
2^{\max \{m,n\}}}\right) ^{\varepsilon _{j}}  \label{4-17-15}
\end{equation}%
for $n,m\in \mathbb{N}$, $k=1,\cdots ,2^{n}$, where%
\begin{equation}
\varepsilon _{j}=\left[ \frac{p-2}{2}-\left( \theta +\beta \right) \right] 
\frac{2j\tilde{N}}{p}-1\text{, \ \ }j=1,2\text{.}  \label{4-21-5}
\end{equation}
\end{proposition}

\begin{proof}
For each fixed $k=1,\cdots ,2^{n}$, consider $f_{m,n,k}^{j}(w)=\left\vert
w_{t_{n}^{k-1},t_{n}^{k}}^{(m+1),j}-w_{t_{n}^{k-1},t_{n}^{k}}^{(m),j}\right%
\vert ^{2\tilde{N}}$ which is continuous on $\boldsymbol{W}$. Thus,
according to the capacity maximal inequality 
\begin{eqnarray*}
&&\text{\textit{Cap}}_{q,N}\left\{ \left\vert
w_{t_{n}^{k-1},t_{n}^{k}}^{(m+1),j}-w_{t_{n}^{k-1},t_{n}^{k}}^{(m),j}\right%
\vert ^{\frac{p}{j}}>C_{\theta }\left( \frac{1}{2^{m}}\right) ^{\beta
}\left( \frac{1}{2^{n}}\right) ^{\theta +1}\right\} \\
&=&\text{\textit{Cap}}_{q,N}\left\{ f_{m,n,k}^{j}>\left[ C_{\theta }^{\frac{j%
}{p}}\left( \frac{1}{2^{m}}\right) ^{\beta \frac{j}{p}}\left( \frac{1}{2^{n}}%
\right) ^{\frac{j}{p}(\theta +1)}\right] ^{2\tilde{N}}\right\} \\
&\leq &C\left[ C_{\theta }^{\frac{j}{p}}\left( \frac{1}{2^{m}}\right)
^{\beta \frac{j}{p}}\left( \frac{1}{2^{n}}\right) ^{\frac{j}{p}(\theta +1)}%
\right] ^{-2\tilde{N}}||f_{m,n,k}^{j}||_{q,N}\text{.}
\end{eqnarray*}%
On the other hand 
\begin{eqnarray*}
&&\text{\textit{Cap}}_{q,N}\left\{ \sum_{k=1}^{2^{n}}\left\vert
w_{t_{n}^{k-1},t_{n}^{k}}^{(m+1),j}-w_{t_{n}^{k-1},t_{n}^{k}}^{(m),j}\right%
\vert ^{\frac{p}{j}}>C_{\theta }\left( \frac{1}{2^{m}}\right) ^{\beta
}\left( \frac{1}{2^{n}}\right) ^{\theta }\right\} \\
&\leq &\sum_{k=1}^{2^{n}}\text{\textit{Cap}}_{q,N}\left\{ \left\vert
w_{t_{n}^{k-1},t_{n}^{k}}^{(m+1),j}-w_{t_{n}^{k-1},t_{n}^{k}}^{(m),j}\right%
\vert ^{\frac{p}{j}}>C_{\theta }\left( \frac{1}{2^{m}}\right) ^{\beta
}\left( \frac{1}{2^{n}}\right) ^{\theta +1}\right\}
\end{eqnarray*}%
It follows from (\ref{4-15-1}) that 
\begin{eqnarray*}
&&\text{\textit{Cap}}_{q,N}\left\{ \sum_{k=1}^{2^{n}}\left\vert
w_{t_{n}^{k-1},t_{n}^{k}}^{(m+1),2}-w_{t_{n}^{k-1},t_{n}^{k}}^{(m),2}\right%
\vert ^{\frac{p}{2}}>C_{\theta }\left( \frac{1}{2^{m}}\right) ^{\beta
}\left( \frac{1}{2^{n}}\right) ^{\theta }\right\} \\
&\leq &\left\{ 
\begin{array}{cc}
C2^{n}\left[ \left( \frac{1}{2^{m}}\right) ^{\frac{2}{p}\beta }C_{\theta }^{%
\frac{2}{p}}\left( \frac{1}{2^{n}}\right) ^{\frac{2}{p}(\theta +1)}\right]
^{-2\tilde{N}}\left( \frac{1}{2^{n}}\right) ^{\tilde{N}}\left( \frac{1}{2^{m}%
}\right) ^{\tilde{N}} & \text{ \ \ for }n\leq m\text{,} \\ 
C2^{n}\left[ \left( \frac{1}{2^{m}}\right) ^{\frac{2}{p}\beta }C_{\theta }^{%
\frac{2}{p}}\left( \frac{1}{2^{n}}\right) ^{\frac{2}{p}(\theta +1)}\right]
^{-2\tilde{N}}\left( \frac{2^{m}}{2^{2n}}\right) ^{2\tilde{N}} & \text{for }%
n>m%
\end{array}%
\right. \\
&\leq &C\left( \frac{1}{2^{m\vee n}}\right) ^{\frac{p-2}{p}2\tilde{N}-\frac{4%
}{p}\tilde{N}\left( \beta +\theta \right) -1}\text{ \ \ for all }n\text{ and 
}m\text{.}
\end{eqnarray*}%
Similarly, for $j=1$ and $n>m$ we have%
\begin{eqnarray*}
&&\text{\textit{Cap}}_{q,N}\left\{ \sum_{k=1}^{2^{n}}\left\vert
w_{t_{n}^{k-1},t_{n}^{k}}^{(m+1),1}-w_{t_{n}^{k-1},t_{n}^{k}}^{(m),1}\right%
\vert ^{p}>C_{\theta }\left( \frac{1}{2^{m}}\right) ^{\beta }\left( \frac{1}{%
2^{n}}\right) ^{\theta }\right\} \\
&\leq &C2^{n}\left[ \left( \frac{1}{2^{m}}\right) ^{\frac{1}{p}\beta
}C_{\theta }^{\frac{1}{p}}\left( \frac{1}{2^{n}}\right) ^{\frac{1}{p}(\theta
+1)}\right] ^{-2\tilde{N}}||f_{1}||_{q,N} \\
&\leq &C2^{n}\left[ \left( \frac{1}{2^{m}}\right) ^{\frac{1}{p}\beta
}C_{\theta }^{\frac{1}{p}}\left( \frac{1}{2^{n}}\right) ^{\frac{1}{p}(\theta
+1)}\right] ^{-2\tilde{N}}\left( \frac{2^{m}}{2^{2n}}\right) ^{\tilde{N}} \\
&\leq &C\left( \frac{1}{2^{n}}\right) ^{\frac{p-2}{p}\tilde{N}-\frac{2}{p}%
\tilde{N}\left( \theta +\beta \right) -1}
\end{eqnarray*}%
which completes the proof.
\end{proof}

\begin{proposition}
Let $p\in (2,3),q\geq 1$, $N\in N$, and $\beta \in (0,\frac{p-2}{2})$. Then,
for any $\varepsilon >0$ there is a constant $C$ depending only on $p,d,q,N$
and $\beta $ such that 
\begin{equation}
\text{\textit{Cap}}_{q,N}\left\{ \rho _{j}(\boldsymbol{w}^{(m+1)},%
\boldsymbol{w}^{(m)})>\left( \frac{1}{2^{m}}\right) ^{\frac{j}{p}\beta
}\right\} \leq C\left( \frac{1}{2^{m}}\right) ^{\varepsilon }
\label{4-22-5a}
\end{equation}%
for all $m\in \mathbb{N}$ and $j=1,2$.
\end{proposition}

\begin{proof}
Choose $\theta >0$ and $\tilde{N}\in \mathbb{N}$ such that $\tilde{N}>N$, 
\begin{equation*}
p-2-\frac{p}{\tilde{N}}>0\text{, }\beta +\theta <\frac{p-2}{2}-\frac{p}{2%
\tilde{N}}
\end{equation*}%
and%
\begin{equation*}
\left[ \frac{p-2}{2}-\left( \theta +\beta \right) \right] \frac{2\tilde{N}}{p%
}-1\geq 2\varepsilon \text{ .}
\end{equation*}%
Then according to Proposition \ref{j2.1}, there is a constant $C$ depending
only on $N,\tilde{N},\theta ,\beta ,d$ and $q$ such that 
\begin{eqnarray*}
&&\text{\textit{Cap}}_{q,\alpha }\left\{ \rho _{j}(\boldsymbol{w}^{(m+1)},%
\boldsymbol{w}^{(m)})^{\frac{p}{j}}>\left( \frac{1}{2^{m}}\right) ^{\beta
}\right\} \\
&\leq &\sum_{n=1}^{\infty }\text{\textit{Cap}}_{q,N}\left\{
\sum_{k=1}^{2^{n}}\left\vert
w_{t_{n}^{k-1},t_{n}^{k}}^{(m+1),j}-w_{t_{n}^{k-1},t_{n}^{k}}^{(m),j}\right%
\vert ^{\frac{p}{j}}>C_{\theta }\left( \frac{1}{2^{m}}\right) ^{\beta
}\left( \frac{1}{2^{n}}\right) ^{\theta }\right\} \\
&\leq &C\sum_{n=1}^{m}\left( \frac{1}{2^{m}}\right) ^{2\varepsilon
}+C\sum_{n\geq m}\left( \frac{1}{2^{n}}\right) ^{2\varepsilon } \\
&\leq &C\left( \frac{1}{2^{m}}\right) ^{\varepsilon }\text{.}
\end{eqnarray*}
\end{proof}

This proposition shows the capacity of $\{\rho _{j}(\boldsymbol{w}^{(m+1)},%
\boldsymbol{w}^{(m)})>2^{-j\beta m/p}\}$ for small $\beta >0$ decays
sub-exponentially in $2^{-m}$ (in contrast with the decay rate in (\ref%
{m23-7a}) which is indeed not a sharp estimate). This is the right order for 
$j=2$. In the case $j=1$ and for capacity Cap$_{2,1}$, this result was
established by M. Fukushima \cite{MR723601}.

\begin{lemma}
Let $p\in (2,3),q\geq 1$, $N\in \mathbb{N}$, $\delta >0$ and $\tilde{N}\in 
\mathbb{N}$ such that 
\begin{equation*}
N\leq \tilde{N}\text{, \ }\tilde{N}\left( 1-\frac{2}{p}\right) -1>0\text{,}
\end{equation*}%
there is a constant $C$ depending only on $N,p,d,\delta ,q\geq 1$ such that 
\begin{equation}
\text{\textit{Cap}}_{q,N}\left\{ \rho _{1}(\boldsymbol{w}^{(m)})>\left(
2^{m}\right) ^{\frac{\delta }{p}}\right\} \leq C\left( \frac{1}{2^{m}}%
\right) ^{\frac{2\delta }{p}\tilde{N}}\text{ \ \ \ \ \ \ }\forall m\in 
\mathbb{N}\text{. }  \label{4-21-8}
\end{equation}
\end{lemma}

\begin{proof}
Choose $\theta >0$ such that%
\begin{equation*}
\tilde{N}\left( 1-2\frac{\theta +1}{p}\right) -1>0\text{.}
\end{equation*}%
Then 
\begin{eqnarray*}
\text{\textit{Cap}}_{q,N}\left\{ \rho _{1}(\boldsymbol{w}^{(m)})^{p}>\left(
2^{m}\right) ^{\delta }\right\} &\leq &\sum_{n=1}^{\infty }\text{\textit{Cap}%
}_{q,N}\left\{ \sum_{k=1}^{2^{n}}\left\vert w_{\frac{k-1}{2^{n}},\frac{k}{%
2^{n}}}^{\left( m\right) ,1}\right\vert ^{p}>C_{\theta }\left( 2^{m}\right)
^{\delta }\left( \frac{1}{2^{n}}\right) ^{\theta }\right\} \\
&\leq &\sum_{n=1}^{\infty }\sum_{k=1}^{2^{n}}\text{\textit{Cap}}%
_{q,N}\left\{ \left\vert w_{\frac{k-1}{2^{n}},\frac{k}{2^{n}}}^{\left(
m\right) ,1}\right\vert ^{p}>C_{\theta }\left( 2^{m}\right) ^{\delta }\left( 
\frac{1}{2^{n}}\right) ^{\theta +1}\right\} \\
&\leq &\sum_{n=1}^{\infty }\sum_{k=1}^{2^{n}}\text{\textit{Cap}}%
_{q,N}\left\{ g_{m,n,k}^{1}>\left[ C_{\theta }^{\frac{1}{p}}\left(
2^{m}\right) ^{\frac{\delta }{p}}\left( \frac{1}{2^{n}}\right) ^{\frac{%
\theta +1}{p}}\right] ^{2\tilde{N}}\right\}
\end{eqnarray*}%
where%
\begin{equation*}
g_{m,n,k}^{1}=\left\vert w_{\frac{k-1}{2^{n}},\frac{k}{2^{n}}}^{\left(
m\right) ,1}\right\vert ^{2\tilde{N}}\text{ .}
\end{equation*}%
Thus, by using the capacity maximal inequality and (\ref{4-17-13}):%
\begin{equation*}
||g_{m,n,k}^{1}||_{q,N}\leq \left\{ 
\begin{array}{cc}
Cq^{\tilde{N}}\left( \frac{2^{m}}{2^{2n}}\right) ^{\tilde{N}} & \text{ \ \
for }n>m\text{,} \\ 
Cq^{\tilde{N}}\left( \frac{1}{2^{n}}\right) ^{\tilde{N}}\text{ \ \ \ } & 
\text{\ \ for }n\leq m\text{,}%
\end{array}%
\right. \text{ \ \ \ \ }
\end{equation*}%
we obtain 
\begin{eqnarray*}
\text{\textit{Cap}}_{q,N}\left\{ \rho _{1}(\boldsymbol{w}^{(m)})^{p}>\left(
2^{m}\right) ^{\delta }\right\} &\leq &C\sum_{n=1}^{m}2^{n}\left[ C_{\theta
}^{\frac{1}{p}}\left( 2^{m}\right) ^{\frac{\delta }{p}}\left( \frac{1}{2^{n}}%
\right) ^{\frac{\theta +1}{p}}\right] ^{-2\tilde{N}}\left( \frac{1}{2^{n}}%
\right) ^{\tilde{N}} \\
&&+C\sum_{n>m}2^{n}\left[ C_{\theta }^{\frac{1}{p}}\left( 2^{m}\right) ^{%
\frac{\delta }{p}}\left( \frac{1}{2^{n}}\right) ^{\frac{\theta +1}{p}}\right]
^{-2\tilde{N}}\left( \frac{2^{m}}{2^{2n}}\right) ^{\tilde{N}} \\
&\leq &C\left( \frac{1}{2^{m}}\right) ^{\tilde{N}\frac{2\delta }{p}%
}\sum_{n=1}^{m}\left( \frac{1}{2^{n}}\right) ^{\tilde{N}\left( 1-2\frac{%
\theta +1}{p}\right) -1} \\
&&+C\left( \frac{1}{2^{m}}\right) ^{\tilde{N}\frac{2\delta }{p}%
}\sum_{n>m}\left( \frac{2^{m}}{2^{n}}\right) ^{\tilde{N}}\left( \frac{1}{%
2^{n}}\right) ^{\tilde{N}\left( 1-2\frac{\theta +1}{p}\right) -1} \\
&\leq &C\left( \frac{1}{2^{m}}\right) ^{\tilde{N}\frac{2\delta }{p}}\text{ .}
\end{eqnarray*}
\end{proof}

\begin{proposition}
Let $p\in (2,3),q\geq 1$, $N\in \mathbb{N}$, and $\beta \in (0,\frac{p-2}{2}%
) $. Then, for any $\varepsilon >0$ there is a constant $C$ depending only
on $p,d,q,N$ and $\beta $ such that 
\begin{equation}
\text{\textit{Cap}}_{q,N}\left\{ \rho _{1}(\boldsymbol{w}^{(m)})\rho _{1}(%
\boldsymbol{w}^{(m+1)},\boldsymbol{w}^{(m)})>\left( \frac{1}{2^{m}}\right) ^{%
\frac{\beta }{p}}\right\} \leq C\left( \frac{1}{2^{m}}\right) ^{\varepsilon }
\label{4-22-6a}
\end{equation}%
for all $m\in \mathbb{N}$.
\end{proposition}

\begin{proof}
Choose $\tilde{N}\in \mathbb{N}$, $\theta >0,\delta >0$ such that%
\begin{equation*}
\tilde{N}\left( 1-2\frac{\theta +1}{p}\right) -1>\varepsilon \text{,}
\end{equation*}%
\begin{equation*}
p-2-\frac{p}{\tilde{N}}>0\text{, }\beta +\theta +\delta <\frac{p-2}{2}-\frac{%
p}{2\tilde{N}}
\end{equation*}%
and%
\begin{equation*}
\left[ \frac{p-2}{2}-\left( \theta +\beta +\delta \right) \right] \frac{2%
\tilde{N}}{p}-1\geq 2\varepsilon \text{ .}
\end{equation*}%
Then, since 
\begin{eqnarray*}
&&\left\{ \rho _{1}(\boldsymbol{w}^{(m)})^{p}\rho _{1}(\boldsymbol{w}%
^{(m+1)},\boldsymbol{w}^{(m)})^{p}>\left( \frac{1}{2^{m}}\right) ^{\beta
}\right\} \\
&\subset &\left\{ \rho _{1}(\boldsymbol{w}^{(m)})^{p}>\left( 2^{m}\right)
^{\delta }\right\} \cup \left\{ \rho _{1}(\boldsymbol{w}^{(m+1)},\boldsymbol{%
w}^{(m)})^{p}>\left( \frac{1}{2^{m}}\right) ^{\beta +\delta }\right\}
\end{eqnarray*}%
so that\newline
\begin{eqnarray*}
&&\text{\textit{Cap}}_{q,N}\left\{ \rho _{1}(\boldsymbol{w}^{(m)})^{p}\rho
_{1}(\boldsymbol{w}^{(m+1)},\boldsymbol{w}^{(m)})^{p}>\left( \frac{1}{2^{m}}%
\right) ^{\beta }\right\} \\
&\leq &\text{\textit{Cap}}_{q,N}\left\{ \rho _{1}(\boldsymbol{w}%
^{(m)})^{p}>\left( 2^{m}\right) ^{\delta }\right\} +\text{\textit{Cap}}%
_{q,N}\left\{ \rho _{1}(\boldsymbol{w}^{(m+1)},\boldsymbol{w}%
^{(m)})^{p}>\left( \frac{1}{2^{m}}\right) ^{\beta +\delta }\right\} \\
&\leq &C\left( \frac{1}{2^{m}}\right) ^{\varepsilon }\text{.}
\end{eqnarray*}
\end{proof}

Putting ((\ref{m29-10}), (\ref{4-22-5a}) and (\ref{4-22-6a}) together we may
conclude the following

\begin{theorem}
Let $p\in (2,3)$. Then for any $\varepsilon >0$, $q\geq 1$ and $N\in \mathbb{%
N}$, there are $\beta >0$, constants $C_{1}>0$ and $C_{2}>0$ depending only
on $p,q,N$, $d$ and $\varepsilon $ such that 
\begin{equation}
\text{\textit{Cap}}_{q,N}\left\{ d_{p}(\boldsymbol{w}^{(m+1)},\boldsymbol{w}%
^{(m)})>C_{1}\left( \frac{1}{2^{m}}\right) ^{\beta }\right\} \leq
C_{2}\left( \frac{1}{2^{m}}\right) ^{\varepsilon }\text{ \ \ }\forall m\in 
\mathbb{N}\text{.}  \label{4-21-9}
\end{equation}
\end{theorem}

We are in a position to prove the main theorem \ref{th5}. By the capacity
version of the Borel-Cantelli lemma, (\ref{4-21-9}) implies that 
\begin{equation*}
A=\left\{ w\in \boldsymbol{W}:\sum_{m=1}^{\infty }d_{p}(\boldsymbol{w}%
^{(m+1)},\boldsymbol{w}^{(m)})=\infty \right\}
\end{equation*}%
is slim, that is, \textit{Cap}$_{q,N}\left\{ A\right\} =0$ for any $q\geq 1$
and $N\in \mathbb{N}$, so that 
\begin{equation*}
\left\{ w\in \boldsymbol{W}:(\boldsymbol{w}^{(m)})\text{ is not Cauchy in }%
G_{p}\Omega (\mathbb{R}^{d})\right\}
\end{equation*}%
is slim, and therefore $\boldsymbol{w}^{(m)}\rightarrow \boldsymbol{w}$ in $%
G_{p}\Omega (\mathbb{R}^{d})$ quasi-surely.




\end{document}